\documentclass[a4paper,12pt]{article}

\usepackage{mymacros}

\newcommand{\sm}{\mathrm{sm}}
\newcommand{\homA}{\Hom}
\newcommand{\mA}{\frakm}
\newcommand{\C}{C}

\newcommand{\HW}{\mathop{HW}\nolimits}

\newcommand{\CF}{\mathop{CF}\nolimits}
\newcommand{\HF}{\mathop{HF}\nolimits}
\newcommand{\crit}{\mathrm{crit}}
\newcommand{\Ll}{{L_{1,0,\lambda}}}
\newcommand{\pl}{{p_{\lambda, \alpha}}}
\newcommand{\ql}{{q_\lambda}}
\newcommand{\rl}{{r_\lambda}}
\newcommand{\Xl}{{X^0_\Lambda}}
\newcommand{\hol}{\operatorname{hol}}
\newcommand{\MA}{\scrA}

\newcommand{\Symp}{\operatorname{Symp}}

\title{Lagrangian torus fibrations and\\
homological mirror symmetry for the conifold}
\author{Kwokwai Chan, Daniel Pomerleano, Kazushi Ueda}
\date{}
\begin{document}

\maketitle

\begin{abstract}
We discuss homological mirror symmetry for the conifold
from the point of view of the Strominger-Yau-Zaslow conjecture.
\end{abstract}

\section{Introduction}

The behavior of strings and branes near the tip of a cone
has been studied extensively in string theory.
The case when the cone is a Gorenstein affine toric 3-fold
is of particular importance,
not only from the point of view of mirror symmetry,
but also for applications to geometric engineering of Seiberg-Witten theory
\cite{MR1467889}
and the AdS/CFT correspondence
\cite{MR1743597}.

Let $Z$ be a Gorenstein affine toric 3-fold
and $\varphi : X \to Z$ be a crepant resolution.
The convex hull $\triangle$ of primitive generators
of one-dimensional cones of the fan
describing $Z$ as a toric variety is a lattice polygon,
which lies on the plane
$$
 \Nbar = \lc n = (n_1, n_2, n_3) \in N \relmid
  n_3 = 1 \rc
$$
under a suitable choice of a coordinate
$N \cong \bZ^3$
on the lattice of one-parameter subgroups
of the dense torus.

If $\triangle$ contains an interior lattice point,
then $X$ is derived-equivalent to the total space $\scK_\frakX$
of the canonical bundle of a 2-dimensional toric Fano stack $\frakX$,
and homological mirror symmetry for $X$ is related
to homological mirror symmetry for $\frakX$
by suspension \cite{Seidel_suspension}.
The case when $\triangle$ does not contain
any interior lattice point is more elusive,
and we discuss such a case in this paper.

Let $Z$ be the {\em conifold},
which is a synonym for a 3-dimensional ordinary double point;
\begin{align*}
Z = \{ (u_1, v_1, u_2, v_2) \in \bC^4 \mid u_1 v_1 = u_2v_2 \}.
\end{align*}
The lattice polygon $\triangle$ for $Z$
is the unit lattice square,
which does not contain any interior lattice points.
The smoothing
\begin{align*}
 Y = \{ (u_1, v_1, u_2, v_2) \in \bC^4 \mid u_1 v_1 = u_2v_2-\epsilon \}
\end{align*}
of the conifold is expected to be mirror
to the small resolution $\varphi : X \to Z$
(cf.~e.g.~\cite{Seidel-Thomas,MR1882328}).

In this paper,
we discuss homological mirror symmetry for the conifold
from the point of view of the Strominger-Yau-Zaslow conjecture \cite{Strominger-Yau-Zaslow}.
To do this,
it is convenient to consider an open subvariety $Y^0$ of $Y$,
which is the complete intersection in
$\bCx \times \bC^4 = \Spec \bC[z, z^{-1}, u_1, u_2, v_1, v_2]$
defined by
\begin{align}
\begin{cases}
 u_1 v_1 = z - a, \\
 u_2 v_2 = z - b.
\end{cases}
 \label{eq:Y_conifold}
\end{align}
Here $a$ and $b$ are distinct non-zero complex numbers,
which we assume to be negative real numbers for simplicity
in this section.

We equip $Y^0$ with the restriction $\omega$
of the symplectic form on $\bCx \times \bC^4$
obtained as the sum of the cylindrical K\"{a}hler form
on $\bCx$ and the Euclidean K\"{a}hler form on $\bC^4$.
Then the map
$$
\begin{array}{cccc}
 \rho : & Y^0 & \to & \bR^3 \\
  & \vin & & \vin \\
  & (z, u_1, v_1, u_2, v_2) & \mapsto &
   \lb \log |z|, \frac{1}{2} \lb |u_1|^2-|v_1|^2 \rb,
    \frac{1}{2} \lb |u_2|^2-|v_2|^2 \rb \rb
\end{array}
$$
is a Lagrangian torus fibration,
whose discriminant loci is given by the disjoint union of two skew lines
$$
 \Gamma = \lc (\log |a|, 0, \lambda_2) \in B \mid \lambda_2\in\bR \rc
  \cup \lc (\log |b|, \lambda_1, 0)\in B \mid \lambda_1\in\bR \rc
$$
as shown in \pref{fg:SYZ_base}.
\begin{figure}[ht]
\centering
\input{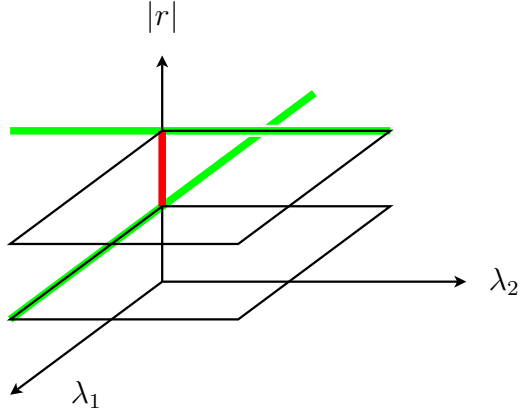}
\caption{The base of the SYZ fibration}
\label{fg:SYZ_base}
\end{figure}

The regular fibers of $\rho$ are special
with respect to the holomorphic volume form
$$
 \Omega = d \log z \wedge d \log u_1 \wedge d \log u_2,
$$
and we will refer to $\rho$ as
the {\em SYZ fibration}.
%
%

The mirror $X^0$ of $Y^0$ is identified
in \cite[Theorem 11.1]{Abouzaid-Auroux-Katzarkov_LFB}
as the complement of a divisor
in the resolved conifold;
\begin{align*}
 X^0 &= X \setminus D, \\
 X &= 
  \scO_{\bP^1}(-1) \oplus \scO_{\bP^1}(-1).
\end{align*}
Here, the divisor $D$ is the pull-back of the divisor
$
 \{ w_1 w_2 = 0 \}
$
on the conifold
\begin{align*}
 Z = \{ (u, v, w_1, w_2) \in \bC^4 \mid u v = (1 + w_1)(1 + w_2) \}
\end{align*}
along the crepant resolution $\varphi : X \to Z$.
The natural projection and
the inclusion of the zero-section
will be denoted by
$\pi : X^0 \to \bP^1$ and
$\iota : \bP^1 \hookrightarrow X^0$
respectively.
%
%
Let $E \subset X^0$ be the image of $\iota$,
which is the exceptional locus of the resolution.
We write $\scO_{X^0}(i) := \pi^* \scO_{\bP^1}(i)$
and $\scO_E(i) := \iota_* \scO_{\bP^1}(i)$ for short. \vskip 10 pt

To a strongly admissible path $\gamma$,
the definition of which we defer to \pref{sc:lag_path},
one can associate an exact non-compact Lagrangian submanifold
$L_\gamma \subset Y^0$,
which is a section of the SYZ fibration $\rho : Y^0 \to \bR^3$.
The {\em SYZ transform}
\cite{Arinkin-Polishchuk_FCFT,Leung-Yau-Zaslow_SLHYM}
of a Lagrangian section of an SYZ fibration
is a holomorphic line bundle
on the mirror,
obtained as a kind of Fourier transform.

\begin{theorem} \label{th:SYZ}
The SYZ transform $\scL_\gamma$
of the Lagrangian section $L_\gamma$
associated with a strongly admissible path
$\gamma : \bR \to \bCx \setminus \Delta$
is the line bundle
$\scO_{X^0}(-w(\gamma))$ on $X^0$.
\end{theorem}
Here $w(\gamma)$ denotes the winding number
defined in \pref{sc:lag_path}.
%
%
Let $\gamma_0$ and $\gamma_1$ be admissible paths
shown in \pref{fg:noncompact_lagrangians2}.
The associated Lagrangian submanifolds of $Y^0$
will be denoted by
$L_0 := L_{\gamma_0}$ and $L_1 := L_{\gamma_1}$,
whose winding numbers are 0 and $-1$ respectively.
Let $\scW$ be the wrapped Fukaya category of $Y^0$
consisting of $L_0$ and $L_1$.

\begin{figure}[ht]
\begin{minipage}[t]{.5 \linewidth}
\centering
\input{noncompact_lagrangians2.pst}
\caption{Non-compact Lagrangians}
\label{fg:noncompact_lagrangians2}
\end{minipage}
\begin{minipage}[t]{.5 \linewidth}
\centering
\input{compact_lagrangians2.pst}
\caption{Compact Lagrangians}
\label{fg:compact_lagrangians2}
\end{minipage}
\end{figure}


\begin{theorem} \label{th:hms}
There is an equivalence
\begin{align} \label{eq:hms}
 D^b \scW \cong D^b \coh X^0
\end{align}
of triangulated categories
sending $L_i$ to $\scO_{X^0}(i)$ for $i = 0, 1$.
\end{theorem}

There is a natural choice
of a pair $(S_0, S_1)$ of Lagrangian 3-spheres in $Y^0$
which are dual to $(L_0, L_1)$; they are $T^2$-fibrations
over the paths shown in \pref{fg:compact_lagrangians2}.

\begin{theorem} \label{th:SYZ0}
The SYZ transforms of the Lagrangian 3-spheres
$S_0$ and $S_1$
are the line bundles
$\scO_E$ and $\scO_E(-1)$
on the exceptional locus $E$
respectively.
\end{theorem}

Let $\scF_0$ be the Fukaya category of $Y^0$
consisting of $S_0$ and $S_1$,
and $\coh_0 X^0$ be the abelian category of coherent sheaves
supported on the exceptional locus
of the resolution $\varphi : X \to Z$.

\begin{theorem} \label{th:hms0}
There is an equivalence
\begin{align} \label{eq:hms0}
 D^b \scF_0 \cong D^b \coh_0 X^0
\end{align}
of triangulated categories
sending $S_0$ and $S_1$
to $\scO_E$ and $\scO_E(-1)$ respectively.
\end{theorem}

This paper is organized as follows:
We review the construction of the SYZ mirror
for the smoothed conifold from \cite{Abouzaid-Auroux-Katzarkov_LFB}
in \pref{sc:SYZ_construction}.
In \pref{sc:lag_path},
we discuss the construction of Lagrangian submanifolds
in $Y^0$
from paths on the $z$-plane.
In \pref{sc:SYZ_transform},
we recall the definition of the SYZ transform
from \cite{Arinkin-Polishchuk_FCFT,Leung-Yau-Zaslow_SLHYM}
and prove Theorems \ref{th:SYZ} and \ref{th:SYZ0}.
In \pref{sc:coh},
we give an explicit description
of the derived category of coherent sheaves
on the resolved conifold.
In \pref{sc:fuk},
we study the wrapped Fukaya category of $Y^0$
and prove \pref{th:hms}.
In \pref{sc:fuk0},
we study $A_\infty$-operations
on vanishing cycles in $Y^0$
and prove \pref{th:hms0}.
In \pref{sc:immersed},
we study Floer cohomology of immersed Lagrangian $S^2 \times S^1$.
In \pref{sc:small},
we discuss extension of the main results of this paper
to more general small toric Calabi-Yau 3-folds.


{\em Acknowledgment} :
K.~C. is supported by a grant from the Research Grants Council of the Hong Kong Special
Administrative Region, China (Project No.~CUHK404412).
D.~P. would like to thank Denis Auroux and Kevin Lin for their very patient explanation of wall-crossing formulas.
K.~U. is supported by JSPS Grant-in-Aid for Young Scientists
No.~24740043.

%

\section{The construction of the SYZ mirror}
 \label{sc:SYZ_construction}

Recall that $Y^0$ is given by the complete intersection
\begin{align}
\begin{split}
 u_1 v_1 &= z - a, \\
 u_2 v_2 &= z - b
\end{split}
\end{align}
in $\bCx \times \bC^4$, where $a$ and $b$ are distinct negative real numbers. Without loss of generality, we assume that $a<b$.
To construct the mirror of $Y^0$, it is also convenient to regard $Y^0$ as the complement of the anticanonical divisor
$$
 H = \{ (z,u_1,v_1,u_2,v_2) \in Y \mid z = 0 \}
$$
in
$$
 Y = \{ (z,u_1,v_1,u_2,v_2) \in \bC^5 \mid u_1v_1 = z - a,\ u_2v_2 = z - b \}.
$$
In the following, we shall briefly review the construction of the mirror for $Y^0$ (or $Y$ with respect to the divisor $H$) following the SYZ approach in \cite{Auroux_MSTD, Auroux_SLFWCMS}; note that our example is a special case of a much more general construction in \cite[Section 11]{Abouzaid-Auroux-Katzarkov_LFB} (see also \cite[Section 4]{Chan-Lau-Leung12} and \cite[Section 5]{Chan_HMSARTD}).

First of all, there is a Hamiltonian $T^2$-action on $(Y^0,\omega)$:
$$
 (e^{is},e^{it})\cdot (z,u_1,v_1,u_1,v_2) = \lb z, e^{is}u_1, e^{-is}v_1, e^{it}u_2, e^{-it}v_2 \rb
$$
whose moment map is given by
$$
 \phi(z,u_1,v_1,u_1,v_2) = \lb \frac{1}{2} \lb |u_1|^2 - |v_1|^2 \rb, \frac{1}{2} \lb |u_2|^2 - |v_2|^2 \rb \rb.
$$
This action extends to $Y$ and preserves the anticanonical divisor $H$.
The SYZ fibration is given by
\begin{align} \label{eq:SYZ_fib}
\begin{array}{cccc}
 \rho : & Y^0 & \to & B := \bR_{>0}\times\bR^2 \\
  & \vin & & \vin \\
  & (z, u_1, v_1, u_2, v_2) & \mapsto &
   \lb |z|, \phi(z, u_1, v_1, u_2, v_2) \rb.
\end{array}
\end{align}
Note that we use $|z|$ here instead of $\log |z|$ and the base is $\bR_{>0}\times\bR^2$ instead of $\bR^3$. This harmless change is more convenient for us because we would like to extend this map to $\rho: Y \to \Bbar := \bR_{\geq0}\times\bR^2$ so that the preimage of the boundary $\{0\}\times\bR^2$ is precisely given by the hypersurface $H$.

Let $\vec{\lambda} = (\lambda_1,\lambda_2)\in \bR^2$ and $r\in \bR_{>0}$. We denote by
$$
 L_{r,\vec{\lambda}} = \{(z, u_1, v_1, u_2, v_2) \in Y \mid |z| = r,\ \phi(z, u_1, v_1, u_2, v_2) = \vec{\lambda}\}
$$
the fiber of $\rho$ over $(r,\vec{\lambda})\in B = \bR_{>0}\times \bR^2$. Consider the double conic fibration $f:Y \to \bC$ given by projection to the $z$-coordinate. Then $L_{r,\vec{\lambda}}$ can be viewed as a fibration, via $f$, over the circle $C_r = \{z\in \bC^\times \mid |z| = r\}$ with generic fiber $T^2$.
The fiber $L_{r,\vec{\lambda}}$ is singular precisely when
\begin{enumerate}
\item[(i)] $r = |a|$ and $\vec{\lambda}=(0,\lambda_2)$; or
\item[(ii)] $r = |b|$ and $\vec{\lambda} = (\lambda_1,0)$;
\end{enumerate}
so the discriminant loci of $\rho$ is the disjoint union of two skew lines
\begin{align} \label{eq:SYZ_disc}
 \Gamma = \lc (|a|, 0, \lambda_2) \in B \mid \lambda_2\in\bR \rc \cup \lc (|b|, \lambda_1, 0)\in B \mid \lambda_1\in\bR \rc.
\end{align}
We denote by $B^\sm := B \setminus \Gamma$ the smooth loci of the base of the SYZ fibration.
When $L_{r,\vec{\lambda}}$ is smooth, it is a \emph{special} Lagrangian torus in $Y^0$ with respect to the symplectic form $\omega$ and the holomorphic volume form
$$
 \Omega = d \log z \wedge d \log u_1 \wedge d \log u_2.
$$

Let
$$
 L'_{r,\lambda_1} = \{(u_1,v_1)\in \bC^2 \mid |u_1v_1+a|=r,\ |u_1|^2-|v_1|^2 = 2\lambda_1\},
$$
and
$$
 L''_{r,\lambda_2} = \{(u_2,v_2)\in \bC^2 \mid |u_2v_2+b|=r,\ |u_2|^2-|v_2|^2 = 2\lambda_2\}.
$$
Via the map $f':\bC^2\to\bC$ given by $(u_1,v_1)\mapsto u_1v_1+a$, we can think of $L'_{r,\lambda_1}$ as an $S^1$-bundle over the circle $C_r$. Similarly, via the map $f'':\bC^2\to\bC$ given by $(u_2,v_2)\mapsto u_2v_2+b$, $L''_{r,\lambda_2}$ can be thought of as an $S^1$-bundle over the same circle. Then $L_{r,\vec{\lambda}}$ is nothing but the fibred product

\[
\begin{CD}
  L_{r,\vec{\lambda}}=L'_{r,\lambda_1}\times_\bC L''_{r,\lambda_2} @> >> L''_{r,\lambda_2}  \\
  @V VV @V f'' VV \\
  L'_{r,\lambda_1} @> f'>> \bC.
\end{CD}
\]\\

To construct the SYZ mirror, we compute the {\em superpotential}
\cite{Cho-Oh,Fukaya-Oh-Ohta-Ono,FOOO_toric_I,Abouzaid-Auroux-Katzarkov_LFB}
which counts Maslov index two holomorphic discs in $Y$ (caution: not $Y^0$!)
with boundary on the Lagrangian torus fibers of $\rho$.
Our arguments are along the same lines as those in \cite{Auroux_MSTD, Auroux_SLFWCMS}.

By composing a disc with the holomorphic map $f:Y \to \bC$ and applying the maximal principle, we see that Maslov index two discs in $(Y, L_{r,\vec{\lambda}})$ are sections of $f$ over the disc bounded by $C_r$. Omitting subscripts for convenience, for $r$ large, the Lagrangian $L = L'\times_\bC L''$ is Hamiltonian isotopic to a Lagrangian of the form
$$
 (S^1(r_1)\times S^1(r_2)) \times_\bC (S^1(r_3)\times S^1(r_4)).
$$
The $S^1(r_1)\times S^1(r_2)$ component bounds two families of Maslov index two discs, which we will denote by $\beta_1'$ and $\beta_2'$, and the $S^1(r_3)\times S^1(r_4)$ component also bounds two families of Maslov index two discs, which we will denote by $\beta_1''$ and $\beta_2''$. Therefore $L'\times_\bC L''$ bounds four families of Maslov index two discs which we will denote by $(\beta_i',\beta_j'')$ for $i,j=1,2$.

Let $z_1,z_2,z_3,z_4$ be the weights corresponding respectively to
$$
 (\beta_1',\beta_1''), (\beta_1',\beta_2''), (\beta_2',\beta_1''), (\beta_2',\beta_2'').
$$
Since
$$
 (\beta_1',\beta_2'') + (\beta_2',\beta_1'') - (\beta_1',\beta_1'') = (\beta_2',\beta_2''),
$$
we have the relation
$$
z_2z_3/z_1 = z_4.
$$
It follows that the superpotential for large $r$ is given by
$$
 W = z_1 + z_2 + z_3 + \frac{z_2z_3}{z_1}.
$$

\begin{remark}
Note that this is exactly the Hori-Vafa superpotential corresponding to the singular toric variety
$$
 Z = \{(u_1,v_1,u_2,v_2)\in\bC^4 \mid u_1v_1 - u_2v_2 = 0\}.
$$
In a sense, we can think of $r$ large as corresponding to some `toric limit'.
\end{remark}

Using the description of $L=L_{r,\vec{\lambda}}$ as a fibred product, it is easy to see that
\begin{proposition}
A Lagrangian torus fiber $L_{r,\vec{\lambda}}$ bounds a nontrivial Maslov index zero holomorphic disc in $Y$ if and only if $r = |a|$ or $r = |b|$. In other words, there are exactly two walls.
\end{proposition}

Recall that we have $a<b<0$ so that $|b| < |a|$. Let $\alpha'$ denote the Maslov index zero disc bounded by the $L'$ factor and $\alpha''$ the one bounded by the $L''$ factor. Also let $w_1$ and $w_2$ be the corresponding weights. When $r$ is small, the Lagrangian torus $L$ is a fibred product of Chekanov tori $L'\times_\bC L''$, with each factor bounding one family of discs $\beta_0'$ and $\beta_0''$ respectively. So $L$ bounds one family of discs with relative homotopy class $(\beta_0',\beta_0'')$. Let $u$ be the weight corresponding to $(\beta_0',\beta_0'')$. We conclude that, for small $r$, the superpotential is simply given by
$$
 W = u.
$$

To analyze the wall-crossing for counting of Maslov index two discs, we first assume that $\lambda_1 > 0$. As $r$ increases and passes through the first wall $r=|b|$, the class $(\beta_0',\beta_0'')$ deforms naturally to
$$
(\beta_1',\beta_0''),
$$
but it may also pick up the Maslov index zero disc $\alpha'$ and deform into
$$
(\beta_1'+\alpha',\beta_0'') = (\beta_2',\beta_0'').
$$
Similarly, assuming $\lambda_2>0$, as $r$ passes through the second wall $r=|a|$, $(\beta_i',\beta_0'')$ naturally deforms to $(\beta_i',\beta_1'')$ but it may also deform to $(\beta_i',\beta_1''+\alpha'') = (\beta_i',\beta_2'')$.

Hence, the wall-crossing formula for the first wall reads
\begin{equation}\label{eq:wc-I}
 u \mapsto \hat{z}_1 (1 + w_1),
\end{equation}
where $w_1 = \hat{z}_3 / \hat{z}_1$, and the wall-crossing formulas for the second wall are given by
\begin{equation}\label{eq:wc-II}
 \hat{z}_1 \mapsto z_1 (1+w_2),\ \hat{z}_3 \mapsto z_3 (1+w_2),
\end{equation}
where $w_2 = z_2 / z_1$. Composing these formulas gives
$$
 u \mapsto z_1 + z_3 + z_2 + \frac{z_2z_3}{z_1},
$$
so the wall-crossing formulas do make the superpotential for $r$ small agree with that for $r$ large.

\begin{remark}
We comment on transversality and orientation for the above moduli spaces of Maslov index two discs.

For transversality, we may apply the argument in \cite[Lemma 7]{AurouxInfinite}. We briefly explain how this argument carries over to our situation. In our situation, explicit calculation shows that the Maslov index two discs avoid the fixed point locus of the natural $T^2$ action on the total space.  Therefore, for any map $u$, we have a short exact sequence
$$ 0 \to u^*\mathcal{L} \to u^*TX \to u^*TX/\mathcal{L} \to 0,$$
where $\mathcal{L}$ is a trivial rank two bundle with real boundary conditions. Let $\bar{u}$ denote the corresponding map to $\mathbb{C}$, which as remarked above is a section over a disc. It follows that surjectivity of the $\bar{\partial}$ operator on sections of $u^*TX$ with boundary conditions $u^*_{|S^1}(TL)$ reduces to that of the $\bar{\partial}$ operator on the quotient bundle $u^*TX/u^*\mathcal{L} \cong \bar{u}^*T\mathbb{C}$ with the corresponding boundary conditions. The surjectivity of the latter operator is well-known.

Similarly, the argument in the proof of \cite[Corollary 8]{AurouxInfinite} (in particular its 4th paragraph, which in turn rely on constructions in \cite[Chapter 8]{Fukaya-Oh-Ohta-Ono} or \cite[Proposition 5.2]{Cho_CT}) adapts directly to determine the signs.
\end{remark}

Although we have assumed that $\lambda_1>0$ and $\lambda_2>0$, the above calculations also work for other cases when $\lambda_1<0$ or $\lambda_2<0$.
So letting $z = \hat{z}_1$ and $v = z_1^{-1}$, the {\em uncompleted} SYZ mirror $\Yv_0$ of the complement $Y^0 = Y \setminus H$ is given by the union of three charts $U_1$, $U_2$ and $U_3$, all algebraically equivalent to $(\bCx)^3$ and equipped with coordinates $(u, w_1, w_2)$, $(z, w_1, w_2)$ and $(v, w_1, w_2)$ respectively. The wall-crossing formulas then tell us that these charts are glued by
\begin{align*}
u \to z(1 + w_1),\ w_1 \to w_1,\ w_2 \to w_2
\end{align*}
from $U_1$ to $U_2$, and by
\begin{align*}
z \to v^{-1}(1 + w_2),\ w_1 \to w_1,\ w_2 \to w_2
\end{align*}
from $U_2$ to $U_3$.

To have a more concrete description of $\Yv_0$, let us consider the singular variety
\begin{align*}
 Z = \{ (u, v, w_1, w_2) \in \bC^4 \mid u v = (1 + w_1)(1 + w_2) \}
\end{align*}
and its crepant resolution
\begin{align*}
 X &= 
  \scO_{\bP^1}(-1) \oplus \scO_{\bP^1}(-1).
\end{align*}
Let $X^0 = X \setminus D$, where $D$ is the pull-back of the divisor
$
 \{ w_1 w_2 = 0 \}
$
on $Z$ along the crepant resolution $\varphi : X \to Z$. Then we can write
$$
 X^0 = \{ (u,v,w_1,w_2,[x_1:x_2]) \in \bC^2\times(\bCx)^2\times\bP^1 \mid ux_2 = (1 + w_1)x_1,\ (1 + w_2)x_2 = vx_1 \}.
$$

Observe that $U_1$ can be embedded into $X^0$ as the chart where $u \neq 0$ and with coordinates $(u, w_1, w_2)$. Similarly, $U_2$ is the chart of $X^0$ where $x_1/x_2 \neq 0$ and with coordinates $(z := x_1/x_2, w_1, w_2)$, while $U_3$ is the chart of $X^0$ where $v \neq 0$ and with coordinates $(v, w_1, w_2)$. It is clear that these charts satisfy the above gluing relations. Now we claim that the union of these charts is precisely given by the complement $X^0 \setminus (C_1 \cup C_2)$ where
\begin{align*}
C_1 & = \{(u,v,w_1,w_2,[x_1:x_2]) \in X^0 \mid u = v = 0, w_1 = -1, [x_1:x_2] = [1:0]\},\\
C_2 & = \{(u,v,w_1,w_2,[x_1:x_2]) \in X^0 \mid u = v = 0, w_2 = -1, [x_1:x_2] = [0:1]\}.
\end{align*}
To see this, just notice that any point with $u \neq 0$ or $v\neq 0$ is covered by $U_1$ and $U_3$ respectively, and any point whose $\bP^1$ coordinate is not equal to $[1:0]$ or $[0:1]$ is covered by $U_2$.

Hence we conclude that the uncompleted SYZ mirror of $Y$
with respect to the anticanonical divisor $H$ is the \emph{Landau-Ginzburg model} $(\Yv_0,W)$ with
total space\footnote{While the three charts of the uncompleted mirror $\Yv_0$ and their gluing are correctly described in the published version of this paper, the explicit formula for $\Yv_0$ there is incorrect and should be modified as described here. We thank Luis Diogo for discussions which lead us to the discovery of this error.}
$$
 \Yv_0 = X^0 \setminus (C_1 \cup C_2)
$$
and superpotential
$$
 W = u.
$$
Note that $\Yv_0$ is an open subvariety of $X^0$, and the superpotential $W$ naturally extends to $X^0$, so the above argument also gives a natural completion of this Landau-Ginzburg model:
\begin{proposition}[{\cite[Section 11]{Abouzaid-Auroux-Katzarkov_LFB}}]
The Landau-Ginzburg model $(X^0,W)$ is the completed, corrected SYZ mirror to $Y$ with respect to the anticanonical divisor $H$. In particular, $X^0$ is the completed, corrected SYZ mirror to $Y^0 = Y \setminus H$.
\end{proposition}

\begin{remark}
It is natural to speculate that the `missing points' $X^0 \setminus \Yv_0$ correspond to singular fibers $L_u := \rho^{-1}(u)$ of the SYZ fibration
$\rho : Y^0 \to \bR^3$, where $u \in \Gamma$ is a point in the discriminant locus. In Section \ref{sc:immersed},
we will try to justify this speculation by some Floer-theoretic computations.
\end{remark}

\section{Lagrangian submanifolds fibred over paths}
 \label{sc:lag_path}

We introduce a class of Lagrangian submanifolds in $Y^0$ which are fibred over paths in the $z$-plane. Let us start with non-compact Lagrangian submanifolds.

\begin{definition}
A smooth path $\gamma : \bR \to \bCx$ on the $z$-plane such that $\lim_{t \to -\infty} |\gamma(t)| = 0$ and $\lim_{t \to \infty} |\gamma(t)| = \infty$ is said to be {\em admissible} if it intersects the interval $\epsilon := [a, b]$ transversally and does not intersect the discriminant $\Delta = \{ a, b \}$ of the double conic fibration $f: Y^0 \to \bCx$. The {\em winding number} $w(\gamma)$ of an admissible path $\gamma$ is defined as its intersection number with $\epsilon$. We choose the orientation so that a path intersecting $\epsilon$ transversally once and in the counterclockwise direction contributes $+1$ to the intersection number.
\end{definition}

Let $\gamma : \bR \to \bCx\setminus \Delta$ be an admissible path. The symplectic fibration $f:Y^0 \to \bCx$ induces a natural horizontal distribution given by symplectic orthogonal to the fiber. Parallel transport with respect to this horizontal distribution gives symplectomorphisms between the smooth fibers of $f$. A 3-dimensional submanifold $L\subset f^{-1}(\gamma)$ is Lagrangian if and only if it is swept by the parallel transport of a Lagrangian cycle in a fiber along $\gamma$ (cf. \cite[Section 5.1]{Auroux_MSTD}). Therefore, by fixing $t_0\in \bR$ and choosing a Lagrangian cycle $A_0$ in the double conic fiber
$$
 f^{-1}(\gamma(t_0)) = (f')^{-1}(\gamma(t_0)) \times (f'')^{-1}(\gamma(t_0)),
$$
one can construct a Lagrangian submanifold $L_{\gamma,A_0}\subset Y$ as the submanifold in $f^{-1}(\gamma)$ swept out by the parallel transport of $A_0$ along $\gamma$.

Notice that the winding number $w(\gamma)$
and the Hamiltonian isotopy class of the Lagrangian submanifold $L_\gamma$
are invariant when we deform $\gamma$ in a fixed isotopy class
relative to the boundary conditions.
In particular, we can always deform $\gamma$
so that $\gamma(t)$ lies on the positive real axis for $t<-T$
for some fixed $T>0$.
Then we consider the direct product
\begin{align*}
A_t:=\{(\gamma(t),u_1,v_1,u_2,v_2) \in f^{-1}(\gamma(t)) \mid u_1,v_1,u_2,v_2 \in \bR\},
\end{align*}
of the real loci (see \pref{fg:fiber_wrapping0})
in the factors $(f')^{-1}(\gamma(t))$ and $(f'')^{-1}(\gamma(t))$
of the double conic fiber $f^{-1}(\gamma(t))$ for each $t<-T$.
The Lagrangian cycle $A_t$ is invariant under symplectic parallel transport for $t<-T$.
We then set
$$
 L_\gamma := L_{\gamma,A_t},
$$
i.e. the submanifold in $Y^0$ swept out by parallel transport of $A_{t_0}$
(for some fixed $t_0<-T$).
This defines a Lagrangian submanifold in $(Y^0,\omega)$
homeomorphic to $\bR^3$.

\begin{definition}
An admissible path $\gamma : \bR \to \bCx \setminus \Delta$
is said to be {\em strongly admissible} if\begin{itemize} \item $|\gamma| : \bR \to \bR_{>0}$ is a strictly increasing function.
\item  The path agrees with a straight line near $z=0$ and outside of some compact set.
\end{itemize}
\end{definition}

\begin{remark} The second condition above is necessary for the purposes of defining wrapped Floer cohomology and in particular for the maximum principle of the appendix to hold. \end{remark}

\begin{proposition}\label{pr:section_SYZ}
Let $\gamma : \bR \to \bCx\setminus\Delta$ be a strongly admissible path. Then the Lagrangian submanifold $L_\gamma$ we define above is a section of the SYZ fibration $\rho: Y^0 \to B$.
\end{proposition}
\begin{proof}
The proof is essentially the same as that of \cite[Proposition 3.4]{Chan-Ueda_DTFHMSAN}. The restriction of the moment map $\phi$ to $A_t$ (for $t$ sufficiently small), which is just the direct product of the real loci (\pref{fg:fiber_wrapping0}), is injective. Since $T^2$ acts fiberwise and it acts by symplectomorphisms on $Y^0$, the symplectic parallel transport induces $T^2$-equivariant symplectomorphisms between fibers of $f$. So the restriction of $\phi$ to a parallel transport of $A_t$ remains injective. Together with the condition that $|\gamma(t)|$ is strictly increasing, we see that $L_\gamma$ is intersecting each fiber of the SYZ fibration $\rho:Y^0\to B$ at one point.
\end{proof}

\begin{remark}
Given a strongly admissible path $\gamma : \bR \to \bCx\setminus\Delta$, we can as well choose any Lagrangian cycle $A_0\subset f^{-1}(\gamma(t_0))$ such that $\phi|_{A_0}$ is an injective map, then the resulting Lagrangian submanifold $L_{\gamma,A_0}$ is also a section of the SYZ fibration.
\end{remark}

An example is given by the path
$$\gamma_0:\bR\to\bC^\times,\ t\mapsto e^t,$$
which runs through the whole positive real axis, which is obviously strongly admissible. The corresponding Lagrangian submanifold $L_0:=L_{\gamma_0}$ is simply the real locus in $Y$ which we choose as the {\em zero-section} of the SYZ fibration.

To construct compact Lagrangian submanifolds in $(Y^0,\omega)$,
we consider {\em bounded paths},
which are smooth paths $\sigma: [0,1] \to \bCx$
starting from the critical value $a$ of one conic fibration
and ending at the critical value $b$ of the other conic fibration.
The fiber product of the Lefschetz thimbles
of each conic fibrations
along a bounded path $\sigma$
gives a Lagrangian submanifold $L_\sigma$ of $Y^0$,
which is a $T^2$ fibration over the bounded path.
One $S^1$-factor collapses to a point on one end
and the other $S^1$-factor collapses to a point on the other end,
so that the total space $L_\sigma$ is homeomorphic to $S^3$.

\begin{definition}
We call a bounded path $\sigma : [0,1] \to \bCx $ going from $b$ to $a$ {\em strongly admissible} if $|\sigma| : [0,1] \to \bR_{>0}$ is a strictly increasing function and $\sigma$ intersects the interval $\epsilon^- := [-b, -a]$ transversally.
\end{definition}

As in \cite{Chan_HMSARTD},
in order to define the SYZ transform,
we need to choose a reference path $\sigma_0$,
relative to which we measure the winding numbers.
Since we have chosen the Lagrangian $L_0$
associated to the positive real axis $\gamma_0$
as the zero-section of the SYZ fibration,
and we would like the Floer cohomology
between the Lagrangians fibered over the two reference paths $\gamma_0$ and $\sigma_0$
to have the correct dimension,
we shall impose the condition
that the reference paths $\gamma_0$ and $\sigma_0$ intersect
transversally at one point (with the correct orientation).
For this reason,
we choose $\sigma_0$ to be the path corresponding to the Lagrangian 3-sphere $S_0$
as shown in \pref{fg:compact_lagrangians2}.

\begin{definition}
The {\em winding number} $w(\sigma)$ of a strongly admissible bounded path $\sigma : [0,1] \to \bCx$ going from $b$ to $a$ is defined to be the winding number of the concatenation of paths $\overline{\sigma}_0 \circ \sigma$ with respect to the counterclockwise isomorphism $\pi_1(\bC^*) \cong \mathbb{Z}$, where $\overline{\sigma}_0$ denotes the path $\sigma_0$ with reversed orientation.
\end{definition}

With this definition, the bounded path $\sigma_1$,
which corresponds to the Lagrangian 3-sphere $S_1$ in \pref{fg:compact_lagrangians2},
has winding number 1.

It is easy to see that the Lagrangian 3-sphere $L_\sigma$
associated with a strongly admissible bounded path $\sigma : [0,1] \to \bCx$
is fibred by $T^2$ over the line segment (the red line in \pref{fg:SYZ_base})
$$
 \ell_0 := (|b|,|a|) \times \{\vec{0}\}
$$
in the base $B$ of the SYZ fibration, and the $T^2$ fiber degenerates to an $S^1$ at both ends $(|a|,\vec{0})$ and $(|b|,\vec{0})$.

\section{SYZ transforms}
 \label{sc:SYZ_transform}

Let $x_1=-\lambda_1$, $x_2=-\lambda_2$ and $x_3$ be affine coordinates (action coordinates)
on the smooth locus $B^\sm$ of the SYZ fibration;
note that $x_1=-\lambda_1$ and $x_2=-\lambda_2$ are globally defined coordinates. We denote by $\Lambda^\vee\subset T^*B^\sm$ the family of lattices locally generated by $dx_1,dx_2, dx_3$, and let
\begin{align*}
\omega_0:=dx_1\wedge d\xi_1 + dx_2\wedge d\xi_2 + dx_3\wedge d\xi_3
\end{align*}
be the standard symplectic structure on the quotient $T^*B^\sm/\Lambda^\vee$ of the cotangent bundle $T^*B^\sm$ by $\Lambda^\vee$, where $(\xi_1, \xi_2, \xi_3)$ denote the fiber coordinates on $T^*B^\sm$. Since we have a global Lagrangian section $L_0$ (the zero-section) of the SYZ fibration $\rho:Y\to B$, there exists a fiber-preserving symplectomorphism \cite{Duistermaat_GAAC}
\begin{align*}
\Theta:(T^*B^\sm/\Lambda^\vee,\omega_0)\overset{\cong}{\longrightarrow}(\rho^{-1}(B^\sm),\omega)
\end{align*}
so that $L_0$ is mapped to the zero section of $T^*B^\sm/\Lambda^\vee$.

We take an open cover $\{U_i\}$ of $B^\sm$ such that each $U_i$ is contractible. As we have seen in \pref{sc:SYZ_construction}, the SYZ mirror $X^0$ is obtained by gluing the open pieces $TU_i/TU_i\cap\Lambda$ together according to the wall-crossing formulas \eqref{eq:wc-I}, \eqref{eq:wc-II} (and then extending by analytic continuation). Let $y_1, y_2, y_3$ be the coordinates on $TB^\sm$ which are dual to the angle coordinates $\xi_1, \xi_2, \xi_3$ on $T^*B^\sm/\Lambda^\vee$. The local complex coordinates on $X^0$ are then given by $w_1=\exp 2\pi(x_1 + \sqrt{-1}y_1)$, $w_2=\exp 2\pi(x_2 + \sqrt{-1}y_2)$ and $\exp 2\pi(x_3 + \sqrt{-1}y_3)$.

Let $L \subset Y^0$ be a Lagrangian cycle,
given as the quotient of a translate of the conormal bundle $N^*S$
of an integral affine linear subspace $S\subset B$
by the lattice $N^*S\cap \Lambda^\vee$, and
equipped with a flat $U(1)$-connection $\nabla$.
The SYZ transform of $(L,Y^0)$ is given by a pair $(C,\check{\nabla})$
consisting of the complex submanifold $C$,
which is given by gluing the open pieces
$$
 T(S\cap U_i)/ T(S\cap U_i) \cap \Lambda
$$
according to the wall-crossing formulas \eqref{eq:wc-I}, \eqref{eq:wc-II},
and a $U(1)$-connection $\check{\nabla}$,
the $(0,2)$-part of the curvature two form of which is trivial
and hence defines a holomorphic line bundle $\check{\mathcal{L}}$ over $C\subset X^0$.
\begin{definition}
We define the {\em SYZ transform} of the Lagrangian submanifold $L$ equipped with the flat $U(1)$-connection $\nabla$ to be the holomorphic line bundle $\check{\mathcal{L}}$ over the complex submanifold $C \subset X^0$.
\end{definition}
We refer the reader to the original papers \cite{Leung-Yau-Zaslow_SLHYM, Arinkin-Polishchuk_FCFT} for more details and the precise formulas; see also \cite{Chan_HMSARTD}.

By Proposition \ref{pr:section_SYZ}, the non-compact Lagrangian submanifold $L_\gamma$ associated with a strongly admissible path $\gamma:\bR\to\bC^\times$ is a section of the SYZ fibration $\rho:Y^0\to B$, so its SYZ transform should produce a holomorphic line bundle over $X^0$. Via the symplectomorphism $\Theta$, we can write $L_\gamma$ as a section of $T^*B^\sm/\Lambda^\vee$
\begin{align*}
L_\gamma=\{(x_1, x_2, x_3, \xi_1, \xi_2, \xi_3)\in T^*B^\sm/\Lambda^\vee \mid \xi_j=\xi_j(x_1,x_2,x_3)\textrm{ for $j=1,2,3$}\},
\end{align*}
where $\xi_j=\xi_j(x_1, x_2, x_3)$ ($j=1,2,3$) are smooth functions on $B^\sm$. The condition that $L_\gamma$ being Lagrangian is then equivalent to saying that the functions $\xi_1,\xi_2,\xi_3$ satisfy the relations
\begin{equation*}
\frac{\partial\xi_j}{\partial x_l}=\frac{\partial\xi_l}{\partial x_j}
\end{equation*}
for $j,l=1,2,3$.

The restriction of the Lagrangian section $L_\gamma$ to an open set $U_i\subset B^\sm$ is transformed to a family of connections $\{\check{\nabla}_{\xi(x)} \mid x\in U_i\}$ which patch together to give a $U(1)$-connection over $U_i$ that can locally be written as
\begin{equation*}
\check{\nabla}_{U_i}=d+2\pi\sqrt{-1}(\xi_1dy_1 + \xi_2dy_2 + \xi_3dy_3)
\end{equation*}
over the open piece $TU_i/TU_i\cap\Lambda \subset X^0$. Since the $(0,2)$-part of the curvature two form for each connection vanishes and the wall-crossing formulas are holomorphic, these connections glue together to give globally a holomorphic line bundle $\check{\mathcal{L}}_\gamma$ over $X^0$.

Notice that the isomorphism class of $\check{\mathcal{L}}_\gamma$ is unchanged when we deform $L_\gamma$ in a fixed Hamiltonian isotopy class (or deforming $\gamma$ in a fixed homotopy class relative to the boundary conditions $\lim_{t \to -\infty} |\gamma(t)| = 0$ and $\lim_{t \to \infty} |\gamma(t)| = \infty$). Therefore, we will regard this as defining the SYZ transform of the Hamiltonian isotopy class of the Lagrangian submanifold $L_\gamma$ as an isomorphism class of holomorphic line bundle over $X^0$.

As an immediate example, the SYZ transformation of the zero section $L_0$ gives the structure sheaf $\mathcal{O}_{X^0}$ over $X^0$.

To compute (the isomorphism class of) the line bundle $\check{\mathcal{L}}_\gamma$, note that the degree of its restriction to the exceptional curve $E\cong\bP^1$ in $X^0$ is given by
\begin{align*}
\deg\check{\mathcal{L}}_\gamma|_E = \int_E\frac{\sqrt{-1}}{2\pi}F_{\check{\nabla}} = -\int_E d\xi_3\wedge dy_3 = -(\xi_3(|b|,\vec{0})-\xi_3(|a|,\vec{0})).
\end{align*}
We have the second equality because $y_1, y_2$ are constant (and $x_i=\lambda_i=0$ for $i=1,2$) on $E$. Hence the isomorphism class of the line bundle $\check{\mathcal{L}}_\gamma$ is completely determined by the increment of the angle coordinate $\xi_3$ on the Lagrangian section $L_\gamma$ from $(0,0,|b|)$ to $(0,0,|a|)$ (which is measured with reference to the path $\gamma_0$).

\begin{proof}[Proof of \pref{th:SYZ}]
Arguing as in the proof of \cite[Theorem 1.1]{Chan-Ueda_DTFHMSAN}, we first deform $\gamma$ so that $\gamma(\log|b|)=-b$ and $\gamma(\log|a|)=-a$ and $\gamma(t)\in \bR_{>0}$ for $t \not\in (\log|b|,\log|a|)$ (up to a re-parametrization if necessary). We then further deform $\gamma|_{(\log|b|, \log|a|)}$ to the concatenation of $\gamma_0|_{(\log|b|, \log|a|)}$ (the positive real axis) with a loop winding around the circle $C_{|a|}=\{z\in\bCx \mid |z|=|a|\}$ for $w(\gamma)$ times. Along $\gamma_0$, the angle coordinate $\xi_3$ is constantly zero, and $\xi_3$ increases by one when we wind around $C_{|a|}$ once in the counterclockwise direction. Hence, the increment $\xi_3(|b|,\vec{0})-\xi_3(|a|,\vec{0})$ is precisely given by the winding number $w(\gamma)$. This completes the proof of Theorem \ref{th:SYZ}.
\end{proof}

Let $\gamma_0$ and $\gamma_1$ be admissible paths shown in \pref{fg:noncompact_lagrangians2} which have winding numbers 0 and $-1$ respectively. Their associated Lagrangian submanifolds are denoted by $L_0 := L_{\gamma_0}$ and $L_1 := L_{\gamma_1}$ respectively. By \pref{th:SYZ}, the SYZ transform of $L_i$ is precisely given by the line bundle $\scO_{X^0}(i)$ for $i=0,1$.

Next we consider a strongly admissible bounded path $\sigma : [0,1] \to \bCx$ going from $b$ to $a$. Recall that the corresponding compact Lagrangian 3-sphere $L_\sigma$ is a $T^2$-fibration over the line segment $\ell_0=(|b|,|a|) \times \{\vec{0}\}$ in the base $B=\bR_{>0}\times \bR^2$ of the SYZ fibration $\rho: Y^0 \to B$ such that the $T^2$-fiber degenerates to an $S^1$ over the endpoints $(|a|,\vec{0})$ and $(|b|,\vec{0})$ of $\ell_0$.

Let $L_\sigma^\circ = L_\sigma \cap \rho^{-1}(\ell_0)$, i.e. $L_\sigma$ with the two $S^1$'s over the end points of $\ell_0$ removed. Then $L_\sigma^\circ$ is (the quotient by a lattice of) a translate of the conormal bundle of $\ell_0$. Recall that the coordinates $w_1,w_2$ on $X^0$ are given by $w_1 = \exp 2\pi(x_1 + \sqrt{-1} y_1)$ and $w_2 = \exp 2\pi(x_2 + \sqrt{-1} y_2)$. We equip $L_\sigma$ with the flat $U(1)$-connection
$$
 \nabla_0 = d - \pi\sqrt{-1} (d\xi_1 + d\xi_2).
$$
Then the SYZ transform of $(L_\sigma,\nabla_0)$ produces the complex submanifold in $X^0$ defined by $x_1=x_2=0$, $y_1=y_2=1/2$ or simply $w_1=w_2=-1$, which is precisely the exceptional locus $E\cong\bP^1 \subset X^0$ (cf. \cite[Section 2]{Chan_HMSARTD}).

We also get the $U(1)$-connection
$$
 \check{\nabla} = d + 2\pi\sqrt{-1} \xi_3(x_3,\vec{0}) dy_3
$$
on $E$ which defines a holomorphic line bundle over $E$ whose degree can be computed as
\begin{align*}
\int_E\frac{\sqrt{-1}}{2\pi}F_{\check{\nabla}} = -\int_E d\xi_3\wedge dy_3 = -(\xi_3(|a|,\vec{0})-\xi_3(|b|,\vec{0})) = -w(\sigma).
\end{align*}
We have the last equality because the increment of the angle coordinate $\xi_3$ is measured relative to the reference path $\sigma_0$ which is computed by the winding number of the loop $\overline{\sigma}_0\circ\sigma$ and this is by definition $w(\sigma)$. This proves the following:
\begin{theorem} \label{th:4.2}
The SYZ transform of the compact Lagrangian 3-sphere $L_\sigma$ associated to a strongly admissible bounded path $\sigma:[0,1] \to \bCx$ is given by the line bundle $\scO_E(-w(\sigma))$ over the exceptional locus $E\subset X^0$.
\end{theorem}

\pref{th:SYZ0} is an immediate consequence of \pref{th:4.2}
since the bounded paths defining $S_0$ and $S_1$ have winding numbers $0$ and $1$ respectively.

\section{Coherent sheaves on the resolved conifold}
 \label{sc:coh}

Let $\bCx$ acts on $\bC^4 = \Spec \bC[x, y, t_1, t_2]$
in such a way that $\alpha \in \bCx$ maps $(x, y, t_1, t_2)$
to $(\alpha x, \alpha y, \alpha^{-1} t_1, \alpha^{-1} t_2)$.
It is convenient to realize the resolved conifold as the quotient
\begin{align} \label{eq:rc1}
 X = (\bC^4 \setminus \Sigma) / \bCx
\end{align}
where $\Sigma := \{ (x, y, t_1, t_2) \in \bC^4 \mid x = y = 0 \}$.
In these coordinates, the morphism
$$
\varphi : X \to  Z = \{ (u, v, w_1, w_2) \in \bC^4 \mid u v = (1 + w_1)(1 + w_2) \}
$$
to the conifold is given by
\begin{align*}
 u = x t_1, \
 v = y t_2, \
 w_1 = x t_2 - 1,
 w_2 = y t_1 - 1.
\end{align*}
%
%
\begin{definition} \label{df:tilting}
An object $\scE$ in a triangulated category $\scT$ is a {\em tilting object} if
\begin{itemize}
 \item
$\scE$ is {\em acyclic} in the sense that
$\Ext^k(\scE, \scE) = 0$ for any $k \ne 0$, and
 \item
$\scE$ is a {\em classical generator},
in the sense that the smallest, thick, triangulated subcategory generated by $\scE$
is all of $\scT$.
\end{itemize}
\end{definition}

Note that any classical generator $\scE$ {\em generates} $\scT$
in the sense that $\Hom^k(\scE, A) = 0$
for some $A \in \scT$ and all $k \in \bZ$
implies $A \cong 0$
(cf.~e.g.~\cite[Section 2.1]{Bondal-van_den_Bergh}).
The proof of the following theorem
can be found in \cite[Lemma 3.3]{Toda-Uehara},
and goes back at least to \cite{Rickard, Bondal_RAACS}:

\begin{theorem} \label{th:tilting}
Let $\scE$ be a tilting object
in the derived category $D^b \coh X$ of coherent sheaves
on a smooth quasi-projective variety $X$.
Then $D^b \coh X$ is equivalent to the bounded derived category
of finitely-generated right modules over $\Hom(\scE, \scE)$.
\end{theorem}

The following is well-known
(cf.~e.g.~\cite{Van_den_Bergh_TFNR}):

\begin{theorem}
 \label{th:conifold_tilting}
The direct sum $\scO_X \oplus \scO_X(1)$ is a tilting object in $D^b \coh X$,
whose endomorphism algebra is described by the quiver
\begin{align*}
\\[3mm]
\begin{psmatrix}[mnode=r]
 \scO
  & &
 \scO(1)
\end{psmatrix}
\psset{shortput=nab,arrows=->,labelsep=3pt,nodesep=3pt}
\ncarc[arcangle=40, offset=4pt]{1,1}{1,3}^{x}
\ncarc[arcangle=40]{1,1}{1,3}_{y}
\ncarc[arcangle=40]{1,3}{1,1}_{t_1}
\ncarc[arcangle=40, offset=4pt]{1,3}{1,1}^{t_2}
 \\[3mm]
\end{align*}
with relations
\begin{align} \label{eq:relations}
 \scI = ( x t_1 y - y t_1 x, x t_2 y - y t_2 x, t_1 x t_2 - t_2 x t_1, t_1 y t_2 - t_2 y t_1).
\end{align}
\end{theorem}

Let $\lc P_{a, i_1, i_2} \rc_{(a, i_1, i_2) \in \bZ \times \bN^2}$
be the basis of
$
 \Hom(\scO_X, \scO_X)
 \cong \Hom(\scO_X(1), \scO_X(1))
 \cong \Gamma(\scO_X)
$
defined by
\begin{align*}
 P_{a, i_1, i_2} &=
\begin{cases}
 u^{-a} w_1^{i_1} w_2^{i_2} & a < 0, \\
 v^a w_1^{i_1} w_2^{i_2} & a \ge 0.
\end{cases}
\end{align*}
Similarly,
we define the bases
$\lc Q_{a, i_1, i_2} \rc_{(a, i_1, i_2) \in (\bZ + \frac{1}{2}) \times \bN^2}$
and
$\lc R_{a, i_1, i_2} \rc_{(a, i_1, i_2) \in (\bZ + \frac{1}{2}) \times \bN^2}$
of
$
 \Hom(\scO_X, \scO_X(1))
$
and
$
 \Hom(\scO_X(1), \scO_X)
$
as
\begin{align*}
 Q_{a, i_1, i_2} &=
\begin{cases}
 x u^{-a-1/2} w_1^{i_1} w_2^{i_2} & a < 0, \\
 y v^{a-1/2} w_1^{i_1} w_2^{i_2} & a \ge 0.
\end{cases}
\end{align*}
and
\begin{align*}
 R_{a, i_1, i_2} &=
\begin{cases}
 t_1 u^{-a-1/2} w_1^{i_1} w_2^{i_2} & a < 0, \\
 t_2 v^{a-1/2} w_1^{i_1} w_2^{i_2} & a \ge 0.
\end{cases}
\end{align*}
respectively.

We have the following elementary algebra calculation.

\begin{proposition}[{cf.~\cite[Proposition 4.5]{Pascaleff_FCMPP}}]
The composition of $P_{a, i_1, i_2}$ is given by
\begin{align} \label{eq:composition1}
 P_{b, j_1, j_2} \cdot P_{a, i_1, i_2}
  &= \sum_{s_1, s_2=0}^k \binom{k}{s_1} \binom{k}{s_2}
   P_{a+b, i_1+j_1+s_1, i_2+j_2+s_2}
\end{align}
where
\begin{align*}
 k =
\begin{cases}
 \min \{ |a|, |b| \} & \text{$a$ and $b$ have different signs}, \\
 0 & \text{otherwise}.
\end{cases}
\end{align*}
The composition of $P_{a, i_1, i_2}$ and $Q_{b, j_1, j_2}$ is given by
\begin{align} \label{eq:composition2}
 Q_{b, j_1, j_2} \cdot P_{a, i_1, i_2}
  &= \sum_{s_1, s_2=0}^k \binom{k}{s_1} \binom{k}{s_2}
   Q_{a+b, i_1+j_1+s_1, i_2+j_2+s_2}
\end{align}
where
\begin{align*}
 k =
\begin{cases}
 \min \{ |a|, |b|-1/2 \} & \text{$a$ and $b$ have different signs}, \\
 0 & \text{otherwise},
\end{cases}
\end{align*}
and similarly for the composition of
$P_{a, i_1, i_2}$ and $R_{b, j_1, j_2}$.
The composition of $Q_{a, i_1, i_2}$ and $R_{b, j_1, j_2}$ is given by
\begin{align} \label{eq:composition3}
 R_{b, j_1, j_2} \cdot Q_{a, i_1, i_2}
  &= \sum_{s_1}^{k_1} \sum_{s_2=0}^{k_2} \binom{k_1}{s_1} \binom{k_2}{s_2}
   P_{a+b, i_1+j_1+s_1, i_2+j_2+s_2}
\end{align}
where
\begin{align*}
 k_1 =
\begin{cases}
 \min \{ |a|-1/2, |b|-1/2 \} + 1 & a < 0 \text{ and } b > 0, \\
 \min \{ |a|-1/2, |b|-1/2 \} & a > 0 \text{ and } b < 0, \\
 0 & \text{otherwise},
\end{cases} \\
 k_2 =
\begin{cases}
 \min \{ |a|-1/2, |b|-1/2 \} & a < 0 \text{ and } b > 0, \\
 \min \{ |a|-1/2, |b|-1/2 \} + 1 & a > 0 \text{ and } b < 0, \\
 0 & \text{otherwise}.
\end{cases}
\end{align*}
\end{proposition}


The mirror $X^0$ is the complement $X \setminus D$
of the divisor $D = \{ w_1 w_2 = 0 \}$ on $X$.

\begin{corollary} \label{cr:conifold_tilting}
The direct sum $\scO_{X^0} \oplus \scO_{X^0}(1)$ is a tilting object in $D^b \coh X^0$.
\end{corollary}

\begin{proof}

The fact that  $\scO_{X^0} \oplus \scO_{X^0}(1)$ is a classical generator follows immediately
from the fact that $\scO_X \oplus \scO_X(1)$ is a classical generator
and the equivalence
$$
 D^b \coh X/ D^b \coh_D X \simto D^b \coh X^0
$$
of triangulated categories \cite[Lemma 2.2]{Orlov_FCIC}.
The acyclicity of $\scO_{X^0} \oplus \scO_{X^0}(1)$
follows from the acyclicity of $\scO_X \oplus \scO_X(1)$
and the description
\begin{align*}
 H^k(\scO_{X^0}(i))
  = \varinjlim \lb
 H^k(\scO_X(i))
  \xto{w_1 w_2} H^k(\scO_X(i))
  \xto{w_1 w_2} H^k(\scO_X(i))
  \xto{w_1 w_2} \cdots \rb
\end{align*}
of the cohomology as a direct limit
\cite[(1.13)]{Seidel_ASNT}.
\end{proof}

The derived category $D^b \coh_0 X$ of coherent sheaves
on $X$ supported on the exceptional locus $E$
of the resolution $\varphi : X \to Z$
is generated by $\scO_E$ and $\scO_E(-1)[1]$,
which are Koszul dual to $\scO_X$ and $\scO_X(1)$
in the sense that
\begin{align*}
 \Hom^0(\scO_X, \scO_E) &= \bC, &
 \Hom^0(\scO_X, \scO_E(-1)[-1]) &= 0, \\
 \Hom^0(\scO_X(1), \scO_E) &= 0, &
 \Hom^0(\scO_X(1), \scO_E(-1)[1]) &= \bC.
\end{align*}
The endomorphism $A_\infty$-algebra
of $\scO_E \oplus \scO_E(-1)$ is Koszul dual
to the endomorphism algebra of $\scO_X \oplus \scO_X(1)$.
A convenient way to describe it is given
by the dimer model
shown in \pref{fg:conifold_bt}.

\begin{figure}[ht]
\centering
\begin{minipage}{.4 \linewidth}
\centering
\input{conifold_bt.pst}
\caption{The dimer model}
\label{fg:conifold_bt}
\end{minipage}
\begin{minipage}{.45 \linewidth}
\centering
\input{conifold_quiver.pst}
\caption{The corresponding quiver}
\label{fg:conifold_quiver}
\end{minipage}
\end{figure}

It is a graph $G$ drawn on the real 2-torus
consisting of two nodes and four edges.
One node is painted in black,
and the other is painted in white.
The dual graph of $G$ is combinatorially identical
to $G$,
and we turn each edge of the dual graph into an arrow
by giving the orientation
such that the white node is on the right of the arrow.
This makes the dual graph of $G$
into the quiver $Q = (V, A)$
shown in Figure \ref{fg:conifold_quiver}
with two vertices $V = \{ 0, 1 \}$
and four arrows $A = \{ x, y, t_1, t_2 \}$.
For each arrow $a$ in the quiver,
there are two paths $p_+(a)$ and $p_-(a)$
from the target of $a$
to the source of $a$;
the former goes around the white node,
and the latter goes around the black node.
Then we can equip the quiver with the relation
such that $p_+(a)$ is equivalent to $p_-(a)$
for all arrows;
$\scI = (p_+(a) - p_-(a))_{a \in A}$.
One can easily see that this relation is identical
to the one in \eqref{eq:relations}.

Now the endomorphism $A_\infty$-algebra of $\scO_E \oplus \scO_E(-1)[1]$
is described as follows
\cite[Definition 2.1 and Proposition 2.2]{Futaki-Ueda_A-infinity}:

\begin{itemize}
 \item
The vertices $0$ and $1$ of $Q$ correspond
to objects $\scO_E$ and $\scO_E(-1)[1]$ respectively.
 \item
For a pair $v$ and $w$ of vertices,
the space of morphisms is given by
$$
 \homA^i(v, w) =
  \begin{cases}
   \bC \cdot \id_v & i = 0 \text{ and } v = w, \\
   \vspan \{ a \mid a : w \to v \} & i = 1, \\
   \vspan \{ a^\vee \mid a : v \to w \} & i = 2, \\
  \bC \cdot \id_v^\vee & i = 3 \text{ and } v = w, \\
   0 & \text{otherwise}.
  \end{cases}
$$
 \item
Non-zero $A_\infty$-operations are
$$
 \mA_2(x, \id_v) = \mA_2(\id_w, x) = x
$$
for any $x \in \homA(v, w)$,
$$
 \mA_2(a, a^\vee) = \id_v^\vee
$$
and
$$
 \mA_2(a^\vee, a) = \id_w^\vee
$$
for any arrow $a$ from $v$ to $w$,
$$
 \mA_k(a_1, \dots, a_k) = a_0.
$$
for any cycle $(a_0, \dots, a_k)$ of the quiver
going around a white node, and
$$
 \mA_k(a_1, \dots, a_k) = - a_0.
$$
for any cycle $(a_0, \dots, a_k)$ of the quiver
going around a black node.
 \item
The pairing
$$
 \la \bullet, \bullet \ra :
  \homA(w, v) \otimes \homA(v, w) \to \bC[3]
$$
defined by
$$
 \la a^\vee, a \ra
  = \la \id_v^\vee, \id_v \ra
  = 1
$$
and zero otherwise
makes the endomorphism $A_\infty$-algebra
into a cyclic $A_\infty$-algebra
of dimension three.
\end{itemize}

To be more explicit,
one has
\begin{align*}
 \Hom^i(\scO_E, \scO_E) &=
\begin{cases}
 \bC \cdot \id_{\scO_E} & i = 0, \\
 \bC \cdot \id_{\scO_E}^\vee & i = 3, \\
 0 & \text{otherwise},
\end{cases} \\
 \Hom^i(\scO_E(-1)[1], \scO_E(-1)[1]) &=
\begin{cases}
 \bC \cdot \id_{\scO_E(-1)[1]} & i = 0, \\
 \bC \cdot \id_{\scO_E(-1)[1]}^\vee & i = 3, \\
 0 & \text{otherwise},
\end{cases} \\
 \Hom^i(\scO_E(-1)[1], \scO_E) &=
\begin{cases}
 \bC \cdot x \oplus \bC \cdot y & i = 1, \\
 \bC \cdot t_1^\vee \oplus \bC \cdot t_2^\vee & i = 2, \\
 0 & \text{otherwise},
\end{cases} \\
 \Hom^i(\scO_E, \scO_E(-1)[1]) &=
\begin{cases}
 \bC \cdot t_1 \oplus \bC \cdot t_2 & i = 1, \\
 \bC \cdot x^\vee \oplus \bC \cdot y^\vee & i = 2, \\
 0 & \text{otherwise},
\end{cases}
\end{align*}
with $A_\infty$-operations
\begin{align*}
 \frakm_3(y, t_1, x) &= - t_2^\vee, &
 \frakm_3(t_2, y, t) &= - x^\vee, &
 \frakm_3(x, t_2, y) &= - t_1^\vee, &
 \frakm_3(t_1, x, t_2) &= - y^\vee, \\
 \frakm_3(y, t_2, x) &= t_1^\vee, &
 \frakm_3(t_1, y, t_2) &= x^\vee, &
 \frakm_3(x, t_1, y) &= t_2^\vee, &
 \frakm_3(t_2, x, t_1) &= y^\vee,
\end{align*}
and
\begin{align*}
 \frakm_2(x, x^\vee) &= \id_{\scO_E}^\vee, &
 \frakm_2(y, y^\vee) &= \id_{\scO_E}^\vee, &
 \frakm_2(s^\vee, s) &= \id_{\scO_E}^\vee, &
 \frakm_2(t_1^\vee, t_1) &= \id_{\scO_E}^\vee, \\
 \frakm_2(t_2, t_2^\vee) &= \id_{\scO_E(-1)[1]}^\vee, &
 \frakm_2(t_1, t_1^\vee) &= \id_{\scO_E(-1)[1]}^\vee, &
 \frakm_2(x^\vee, x) &= \id_{\scO_E(-1)[1]}^\vee, &
 \frakm_2(y^\vee, y) &= \id_{\scO_E(-1)[1]}^\vee.
\end{align*}
All the other non-zero $A_\infty$-operations just say
that $\id_{\scO_E}$ and $\id_{\scO_E(-1)[1]}$ are
the identity elements for $\frakm_2$.

\section{Wrapped Fukaya category}
 \label{sc:fuk}

We prove \pref{th:hms} in this section. For technical reasons, Floer theory on a non-compact fibration such as the one we are considering requires a modification of the symplectic form so that
\begin{itemize}
\item the symplectic monodromy is trivial along the horizontal boundary of the fibration.
\item the flow of the Hamiltonians $H_i$ below fiber over the flow in the base for Hamiltonian vector-field of $H_b$ in the base.
\end{itemize}
We refer the reader to the appendix for a more detailed discussion of the geometric setup for Floer cohomology of fibrations.
We set $a = \sqrt{-1}$ and $b = -\sqrt{-1}$
as in \pref{fg:base_wrapping0} for convenience in this section.
We take wrapping Hamiltonians of the forms
\begin{align}
 H_i = H_b + H_{f_1, i} + H_{f_2, i},
\end{align}
where the Hamiltonian $H_b$ is an admissible Hamiltonian in the base (see the appendix for this definition) and wraps the $z$-plane
as shown in \pref{fg:base_wrapping1}, and
the fiber Hamiltonians $H_{f_i, i}$ are admissible Hamiltonians in the fiber which wrap
the fiber either as in \pref{fg:fiber_wrapping1}
or \pref{fg:fiber_wrapping2}.
We assume that each $H_i$ is \emph{Lefschetz admissible} in the sense of \pref{sc:lefschetz_WFC} or McLean \cite{MR2497314}.
Let $\phi_t : Y^0 \to Y^0$ be the time $t$ flow
by the wrapping Hamiltonian $H_i$. The {\em wrapped Floer cohomology} is defined as
$$
 \Hom_{\scW_i}(L_j, L_k)
  = \lim_{t \to \infty} \Hom_\scF(\phi_t L_j, L_k)
$$
where $\Hom_\scF(\phi_t L_j, L_k)$ is the ordinary Floer cohomology.

\begin{remark}
Our choice of Hamiltonian is slightly different from that
in \cite{Abouzaid-Seidel_OSAVF},
but is very suitable for analyzing fibrations.
In the appendix, we provide some details concerning wrapped Floer cohomology as well as the relationship between the two approaches.
It is also important to note that,
while we don't construct $A_\infty$-operations on our wrapped Floer cohomology,
all of our wrapped Floer groups are concentrated in degree zero
and thus any such enhancement would actually be quasi-isomorphic
to its cohomology algebra.
\end{remark}

%

\begin{figure}[ht]
\begin{minipage}[b]{.5 \linewidth}
\centering
\input{base_wrapping0.pst}
\caption{The paths on the base}
\label{fg:base_wrapping0}
\end{minipage}
\begin{minipage}[b]{.5 \linewidth}
\centering
\input{fiber_wrapping0.pst}
 \caption{The Lagrangian on the fiber}
\label{fg:fiber_wrapping0}
\end{minipage}
 \end{figure}

 \begin{figure}[ht]
 \begin{minipage}[b]{.5 \linewidth}
\centering
\input{base_wrapping1.pst}
 \caption{Wrapping the base}
 \label{fg:base_wrapping1}
 \end{minipage}
 \begin{minipage}[b]{.24 \linewidth}
 \centering
 \input{fiber_wrapping1.pst}
\caption{Wrapping the fiber by $H_{f,1}$}
  \label{fg:fiber_wrapping1}
 \end{minipage}
 \begin{minipage}[b]{.24 \linewidth}
\centering
\input{fiber_wrapping2.pst}
 \caption{Wrapping the fiber by $H_{f,2}$}
 \label{fg:fiber_wrapping2}
 \end{minipage}
 \end{figure}

\begin{proposition} \label{pr:ring_isom}
There is a ring isomorphism
\begin{align} \label{eq:ring_isom}
 \bigoplus_{i, j=0}^1 \Hom_{\scW_1}(L_i, L_j)
 \simto
 \bigoplus_{i,j=0}^1 \Hom \lb \scO_X(i), \scO_X(j) \rb.
\end{align}
\end{proposition}

\begin{proof}
Let us first consider the composition
\begin{align} \label{eq:composition2-1}
 \Hom(\phi_n(L_0), L_0) \otimes \Hom(\phi_{m+n}(L_0), \phi_n(L_0))
  \to \Hom(\phi_{m+n}(L_0), L_0).
\end{align}

In the appendix, the product in wrapped Floer cohomology is defined using solutions to a perturbed holomorphic curve equation. The argument of \cite[Proposition 7.2]{Pascaleff_FCMPP}
allows us to show that, in this situation, this is equivalent to the usual product in Lagrangian Floer theory which counts $J$-holomorphic triangles with boundary on  $L_0$, $\phi_n(L_0)$, and $\phi_{m+n}(L_0)$.

The intersection points in $\phi_n(L_0) \cap L_0$
can be labeled as $p_{a, i_1, i_2}$
as in \pref{fg:intersection1}.
We view the $z$-plane as a cylinder, which is obtained by identifying the horizontal edges
of the rectangle in \pref{fg:intersection1}.
We choose a coordinate on the rectangle in such a way that
the top right and the bottom left corners
have coordinates $(1, 1)$ and $(-1, -1)$ respectively.

Intersections between the Lagrangians $\phi_n(L_0)$ and $L_0$ are parameterized by triplets of integers $(a,i_1,i_2)$.
The integer $a \in [-n+1,n-1]$ parametrizes the intersection point of the $z$-projections $\sigma_n(\gamma_0)$ and $\gamma_0$
of the Lagrangians $\phi_n(L_0)$ and $L_0$. The integers $i_1$ and $i_2$
in $[0, \lfloor (n-|a|)/2 \rfloor]$ parametrize the intersection points
on the fiber just as in \cite[Section 3.3.4]{Pascaleff_FCMPP}.

\begin{figure}[htbp]
 \centering
\input{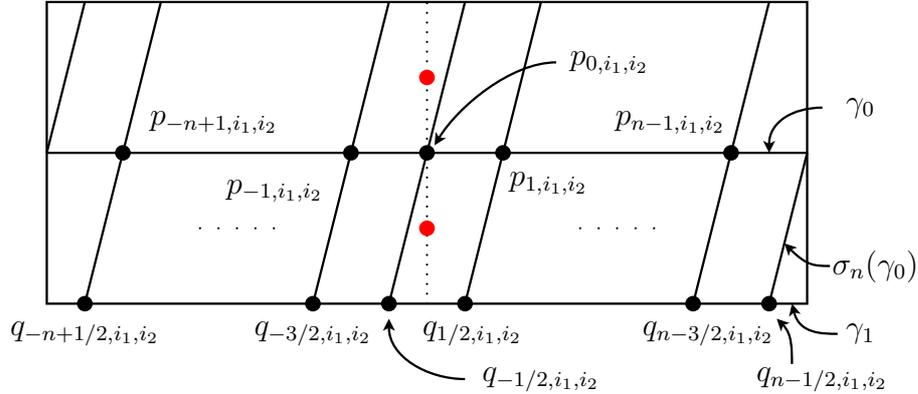}
\caption{Intersections of Lagrangians}
 \label{fg:intersection1}
\end{figure}

Our arguments will be based upon the following adaptation of Pascaleff's theorem
\cite[Proposition 4.4]{Pascaleff_FCMPP} to this setting.
Its proof follows \emph{mutatis-mutandis} from Pascaleff's paper.

\begin{lemma} \label{lm:Pascaleff}
Let $L$, $L'$, and $L''$ be Lagrangian submanifolds of $Y^0$
fibered over paths $\gamma$, $\gamma'$ and $\gamma''$ in $\bCx$.
Assume that a holomorphic triangle $u : D^2 \to \bCx$
bounded by $\gamma$, $\gamma'$ and $\gamma''$
with vertices
$o \in \gamma \cap \gamma'$,
$o' \in \gamma' \cap \gamma''$ and
$o'' \in \gamma \cap \gamma''$
intersects the discriminants $a$ and $b$ in $\bCx$
exactly $d_1$ and $d_2$ times respectively.
Then holomorphic sections over $u$ contributes to the triangle product
$
 \Hom(L', L'') \otimes \Hom(L, L') \to \Hom(L, L'')
$
as
$$
 \frakm_2(o'_{j_1, j_2}, o_{i_1, i_2})
  = \sum_{s_1=0}^{d_1} \sum_{s_2 = 0}^{d_2}
   \binom{d_1}{s_1} \binom{d_2}{s_2} o''_{i_1+j_1+s_1, i_2+j_2+s_2},
$$
where $o_{i_1,i_2} \in L \cap L'$ is the intersection point
above $o \in \gamma \cap \gamma'$,
which is the $i_1$-th one from the bottom in the $u_1v_1$-direction
and the $i_2$-th one from the bottom in the $u_2 v_2$-direction.
\end{lemma}

The universal cover of the cylinder
in \pref{fg:intersection1}  is an infinite strip
$
 \{ (s, t) \in \bR^2 \mid -1 \le s \le 1 \}.
$
A lift of the $z$-projection $\gamma_{0,n}$
of the wrapped Lagrangian $\phi_n(L_0)$
to the universal cover
is given by a line with slope $n$,
passing through $(0, k)$ with $k \in \bZ$.
The discriminants of the conic fibrations are given
by $(0, 1/4)$ and $(0, -1/4)$ respectively.
The projection of the intersection point
$p_{b, j_1, j_2} \in \Hom(\phi_n(L_0), L_0)$
has the $s$-coordinate $b/n$,
and we choose the lift to the universal cover
to be $(b/n, 0)$.

Consider the lift of $\gamma_{0,n}$
passing through $(b/n, 0)$. The induced lift of the intersection point
corresponding to $p_{a, i_1, i_2} \in \Hom(\phi_{m+n}(L_0), \phi_n(L_0))$
will have coordinate $(a/m, n a / m - b)$.
If we then take the lift of $\gamma_{0,m+n}$
passing though this point, it intersects with the lift of $\gamma_0$
at $((a+b)/(m+n), 0)$
as shown in \pref{fg:triangle1} or \pref{fg:triangle2}
depending on the order of $a$ and $b$.

\begin{figure}[htbp]
\begin{minipage}[b]{.5 \linewidth}
 \centering
 \input{triangle1.pst}
\caption{The case $a<b$}
\label{fg:triangle1}
\end{minipage}
\begin{minipage}[b]{.5 \linewidth}
\centering
\input{triangle2.pst}
\caption{The case $b<a$}
\label{fg:triangle2}
\end{minipage}
\end{figure}

In either case,
one can see from \pref{fg:triangle3}
or \pref{fg:triangle4}
that the triangle hits both of the discriminants
$(0, -1/4+\bZ)$ and $(0, 1/4+\bZ)$ $k$ times,
where $k$ is $\min \{ |a|, |b| \}$ if $a$ and $b$ has different signs
and $0$ otherwise.
Then one has
\begin{align} \label{eq:composition4}
 \frakm_2(p_{b, j_1, j_2}, p_{a, i_1, i_2})
  &= \sum_{s_1, s_2=0}^k \binom{k}{s_1} \binom{k}{s_2}
   p_{a+b, i_1+j_1+s_1, i_2+j_2+s_2}
\end{align}
by Pascaleff's formula,
in agreement with \eqref{eq:composition1}.

\begin{figure}[htbp]
\begin{minipage}[b]{.5 \linewidth}
\centering
\input{triangle3.pst}
\caption{The case $0 < b < -a$}
 \label{fg:triangle3}
 \end{minipage}
\begin{minipage}[b]{.5 \linewidth}
\centering
 \input{triangle4.pst}
\caption{The case $0 < -a < b$}
\label{fg:triangle4}
\end{minipage}
\end{figure}

Next we consider the composition
$$
 \Hom(\phi_n(L_1), L_0) \otimes \Hom(\phi_{m+n}(L_0), \phi_n(L_1))
  \to \Hom(\phi_{m+n}(L_0), L_0).
$$
A lift of the $z$-projection $\gamma_{1,n}$
of the wrapped Lagrangian $\phi_n(L_1)$
to the universal cover
is given by a line with slope $n$
passing through $(0, k+1/2)$ with $k \in \bZ$.
The intersections of the curves
on the $z$-planes are as in \pref{fg:triangle1}
or  \pref{fg:triangle2} again,
with $\gamma_{0,n}$ replaced with $\gamma_{1,n}$
and $a$ and $b$ being half-integers.
One can see
from \pref{fg:triangle3}  and \pref{fg:triangle4}
that the triangle hits
the discriminants at $(0, -1/4+\bZ)$
and $(0, 1/4+\bZ)$
for $k_1$ times and $k_2$ times respectively,
where
$k_1 = \min \{ |a|-1/2 , |b|-1/2 \}+1$ and
$k_2 = \min \{ |a|-1/2 , |b|-1/2 \}$
if $a$ and $b$ have different signs,
and $k_1=k_2=0$ otherwise.
Then one has
\begin{align} \label{eq:composition2-4}
 \frakm_2(r_{b, j_1, j_2}, q_{a, i_1, i_2})
  &= \sum_{s_1=0}^{k_1} \sum_{s_2=0}^{k_2} \binom{k}{s_1} \binom{k}{s_2}
   p_{a+b-1, i_1+j_1+s_1, i_2+j_2+s_2}.
\end{align}
This is in complete agreement with \eqref{eq:composition3}. Other compositions can be calculated similarly, and \pref{pr:ring_isom} is proved.
\end{proof}

The choice of partial wrapping function $H_{f_i,1}$ corresponds to the fact that the mirror of the resolved conifold $X$ is in fact the Landau-Ginzburg model $(Y^0, u_1+u_2)$. See \cite{Abouzaid-Auroux-Katzarkov_LFB} for more discussion. Since wrapping by $H_{f_i, 2}$ corresponds to multiplication by $w_i$,
one obtains the following:

\begin{corollary} \label{cr:ring_isom}
There is a ring isomorphism
\begin{align} \label{eq:ring_isom2}
 \bigoplus_{i, j=0}^1 \Hom_{\scW_2}(L_i, L_j)
 \simto
 \bigoplus_{i,j=0}^1 \Hom \lb \scO_{X^0}(i), \scO_{X^0}(j) \rb.
\end{align}
\end{corollary}

\pref{th:hms} is an immediate consequence of
\pref{cr:conifold_tilting},
\pref{cr:ring_isom}, and
\pref{th:appendix2}.

\section{Mirror symmetry for vanishing cycles}
 \label{sc:fuk0}

For a path $\gamma : [0,1] \to \bCx$ on the $z$-plane,
the union
\begin{align} \label{eq:compact_lag}
 S_\gamma := \bigcup_{t \in [0,1]} \lc (\gamma(t), u_1, v_1, u_2, v_2) \in Y^0
  \relmid |u_1|=|v_1|, \ |u_2|=|v_2| \rc
\end{align}
gives a compact Lagrangian submanifold of $Y^0$,
which has boundaries in general.
Let $S_0$ and $S_1$ be Lagrangian 3-spheres in $Y^0$,
which are obtained in this way
from the paths shown in \pref{fg:compact_lagrangians2}.
We prove \pref{th:fuk0} below in this section.
\pref{th:hms0} follows immediately
since $D^b \coh_0 X^0$ is generated by $\scO_E$ and $\scO_E(-1)$
as a triangulated category.

\begin{theorem} \label{th:fuk0}
Let $\scF_0$ be the Fukaya category of $Y^0$
consisting of $S_0$ and $S_1$.
Then $\scF_0$ is quasi-equivalent to the full subcategory
of (the dg enhancement of) $D^b \coh X^0$
consisting of $\scO_E$ and $\scO_E(-1)$.
\end{theorem}

\begin{proof}
First note that the union $S_0 \cup S_1$ is {\em exact}
in the sense that the symplectic form $\omega$ vanishes
on $\pi_2(Y^0, S_0 \cup S_1)$.
To see this,
one can look at the exact sequence
$$
 \cdots
  \to \pi_2(Y^0)
  \to \pi_2(Y^0, S_0 \cup S_1)
  \to \pi_1(S_0 \cup S_1)
  \to \pi_1(Y^0)
  \to \cdots
$$
of homotopy groups;
the symplectic form $\omega$ vanishes
on the image of $\pi_2(Y^0)$
since $Y^0$ is an exact symplectic manifold,
whereas the group $\pi_1(S_0 \cup S_1) \cong \bZ$ injects
to $\pi_1(Y^0)$.

The exactness of $S_0 \cup S_1$ allows us
to use the chain model for the Fukaya category of the plumbing
by Abouzaid \cite[Appendix A]{MR2786590}.
Let $Q_1$ and $Q_2$ be a pair of graded exact Lagrangian submanifolds
in an exact symplectic manifold equipped with a trivialization of the canonical bundle.
Assume that $Q_1$ and $Q_2$ intersect cleanly
along a submanifold $B$,
which consists of $r$ connected components
$B^1, \ldots, B^k$;
$$
 B = Q_1 \cap Q_2, \qquad
 B = B^1 \sqcup \cdots \sqcup B^r.
$$
Since $Q_1$ and $Q_2$ are Lagrangian submanifolds
intersecting cleanly along $B$,
the normal bundles $N_{Q_1} B$ and $N_{Q_2} B$ are isomorphic
as real vector bundles.
Choose closed tubular neighborhoods $N_i$ of $Q_i$
and triangulations $\scQ_i$ of $Q_i$
such that $\scQ_i$ induce triangulations $\scN_i$ of $N_i$
and the isomorphism $N_{Q_1} B \cong N_{Q_2} B$
induces a cellular homeomorphism
$\scN_1 \cong \scN_2$.
Let $\scN = \scN^1 \sqcup \cdots \sqcup \scN^r$
be the decomposition of the abstract simplicial complex
$\scN = \scN_1 \cong \scN_2$
into connected components.
Then the chain model for the Fukaya category
consisting of $Q_i$ is given by
\begin{align*}
 \Hom(Q_i, Q_i) &= \C^*(\scQ_i), \\
 \Hom(Q_1, Q_2) &= \bigoplus_{k=1}^r \C^*(\scN^k)[m_k], \\
 \Hom(Q_2, Q_1) &= \bigoplus_{k=1}^r \C^*(\scN^k, \partial \scN^k)[-m_k].
\end{align*}
Here integers $m_k$ comes from the gradings of the Lagrangian
submanifolds.

Now we apply this construction
to the case
when $Q_1 = S_0 \cong \bS^3$, $Q_2 = S_1 \cong \bS^3$,
$B = B^1 \sqcup B^2 = \bS^1 \sqcup \bS^1$ and
$\scN^k \cong \bD^2 \times \bS^1$;
\begin{align*}
 \Hom(S_i , S_i) &\cong \C^*(\bS^3), \\
 \Hom(S_1,S_0) &\cong \C^*(\bD^2 \times \bS^1)[-1] \oplus \C^*(\bD^2 \times \bS^1)[-1], \\
 \Hom(S_0,S_1) &\cong \C^*(\bD^2 \times \bS^1, \partial \bD^2 \times \bS^1)[1]
  \oplus \C^*(\bD^2 \times \bS^1,\partial \bD^2 \times \bS^1)[1].
\end{align*}
In this formula, cochains denote simplicial cochains with respect to a suitable triangulation.
We have chosen the gradings on $S_0$ and $S_1$ in such a way that
$m_1 = m_2 = -1$.
We view each copy of $\bS^3$ via its Hopf decomposition
$$
 \bS^3 = \bD^2 \times \bS^1 \cup_{\bT^2} \bD^2 \times \bS^1.
$$
In our example,
we can work with the smaller cellular model described below,
which can be easily seen to be that we get a quasi-isomorphic dg-category
if we choose to view each $\bD^2$ as a two simplex $\Delta_2$ and
$\bS^1$ as the union of three one simplices $\Delta_1$ in the usual way.

We have the cochain models
\begin{align*}
 \C^*(\scN^1)
  &= \vspan_{\bC}
\begin{Bmatrix}
  & &\rnode{x}{x} & & \rnode{xy}{xy} & \\
  1 & & & & & & \rnode{yz}{yz} \\
  & & y & & \rnode{z}{z}
\end{Bmatrix}, \\
\psset{nodesep=3pt,arrows=->,shortput=nab}
\ncline{x}{z}
\ncline{xy}{yz}
%
 \C^*(\scN^2)
  &= \vspan_{\bC}
\begin{Bmatrix}
  & &\rnode{x}{x} & & \rnode{xy}{xy} & \\
  1 & & & & & & \rnode{yw}{yw} \\
  & & y & & \rnode{w}{w}
\end{Bmatrix},
\psset{nodesep=3pt,arrows=->,shortput=nab}
\ncline{x}{w}
\ncline{xy}{yw}
\end{align*}
for $\C^*(\scN^k) = \C^*(\bD^2 \times \bS^1)$.
Arrows show the differential
in such a way that $\bsd(x) = z$ and similarly for other arrows.
The cohomological degrees are given by
$\deg(x)=\deg(y)=1$ and $\deg(z)=\deg(w)=2$.
The elements $w$ and $z$ are the cellular cochains which are dual to the disc $\bD^2$
as shown in \pref{fg:hopf}.
We use the same symbols $x$ and $y$ for those generators which will be identified under the Hopf gluing.
In the first copy of $\bD^2 \times \bS^1$,
$x$ is the cochain dual to the boundary of $\bD^2$ and $y$ is the cellular cochain dual to the $\bS^1$ factor.
The roles of these cochains are reversed in the second copy of $\bD^2 \times \bS^1$.
\begin{figure}[h]
\centering
\input{hopf.pst}
\caption{Hopf decomposition of $\bS^3$}
\label{fg:hopf}
\end{figure}
The chain model for one copy of $\C^*(\partial \bD^2 \times \bS^1)$ is given
by
$$
 \C^*(\partial \bD^2 \times \bS^1)
  = \vspan_{\bC}
\begin{Bmatrix}
  & & x\\
  1 & & & & xy \\
  & & y
\end{Bmatrix},
$$
and similarly for the other copy.
Accordingly, we have the chain model
\begin{align*}
 \C^*(\bS^3)
  = \vspan_{\bC}
\begin{Bmatrix}
 && && \rnode{xy}{xy} \\
 && \rnode{x}{x} && && \rnode{yz}{yz} \\
 1 && && \rnode{z}{z} \\
 && \rnode{y}{y} && && \rnode{xw}{xw} \\
 && && \rnode{w}{w}
\end{Bmatrix},
\psset{nodesep=3pt,arrows=->,shortput=nab}
\ncline{x}{z}
\ncline{y}{w}
\ncline{xy}{yz}
\ncline{xy}{xw}
\end{align*}
for $\C^*(\bS^3)$
where $\bsd(x y) = y z + xw$.
Using these basic models,
we construct the chain level model for the category as follows:
For $\C^*(\bS^3)$,
we take the above cochain algebra.
For the other groups,
we preserve the letters corresponding to the above models
to make clear the geometric origins of the generators and use $\overrightarrow{m}$
to denote a morphism in $\Hom(S_1,S_0)$ and
$\overleftarrow{m}$ to denote a morphism in $\Hom(S_0,S_1)$.
For $\Hom(S_0,S_1)$ we take as a basis:
$$ \overleftarrow{z} \quad  \overleftarrow{yz} \quad \overleftarrow{w} \quad  \overleftarrow{xw} , \quad \bsd=0.$$
We have that $\mathrm{Hom(S_1,S_0)}$ is the sum of two complexes,
$$ \overrightarrow{u_{1}} \quad \overrightarrow{x_{1}} \quad \overrightarrow{y_{1}} \quad \overrightarrow{z_{1}} \quad \overrightarrow{xy_{1}} \quad \overrightarrow{yz_{1}}, $$
$$ \overrightarrow{u_{2}} \quad \overrightarrow{x_{2}} \quad \overrightarrow{y_{2}} \quad \overrightarrow{w_{2}} \quad \overrightarrow{xy_{2}} \quad \overrightarrow{xw_{2}}. $$
The differential is as in the model for $\C^*(\bD^2 \times \bS^1)$.
Compositions are the natural ones
described in \cite[Section 2.1]{MR2786590}.

\begin{lemma} \label{lm:m_3}
We have
\begin{align*}
\m_3(\overrightarrow{u_1},\overleftarrow{z},\overrightarrow{u_2}) &=\overrightarrow{x_{2}}, \\
\m_3(\overleftarrow{w},\overrightarrow{u_1},\overleftarrow{z}) &= \overleftarrow{yz}, \\
\m_3(\overrightarrow{u_2}, \overleftarrow{w}, \overrightarrow{u_1}) &= -\overrightarrow{y_{1}}, \\
\m_3(\overleftarrow{z}, \overrightarrow{u_2},\overleftarrow{w}) &= -\overleftarrow{xw}.
\end{align*}
The other $m_3$'s are determined by the natural cyclic Calabi-Yau structure.
\end{lemma}

\begin{proof}
Given a dga $(V,\bsd)$,
we choose
\begin{itemize}
 \item
an injection $\mathrm{i}: \mathbb{H}^*(V) \to V$,
 \item
a projection operator $\mathrm{P}: V \to \mathrm{i}(\bH^*(V))$
such that
$
 \mathrm{P}|_{\mathrm{i}(\mathbb{H}^*(V))}
  = \id_{\mathrm{i}(\mathbb{H}^*(V))},
$
and
 \item
a chain homotopy $\mathrm{Q}$ such that
$
 \id - [\bsd,\mathrm{Q}] = \mathrm{P}.
$
\end{itemize}
Then we define a series of linear maps
$$ \lambda_n : V^{\otimes n} \to V$$
by setting
$$ \lambda_2(v_1,v_2)= v_1 \cdot v_2$$
and inductively define
$$ \lambda_n(v_1, \cdots , v_n) := (-1)^{n-1}
[\mathrm{Q}\lambda_{n-1}(v_1, \cdots , v_{n-1})] v_n - (-1)^{n \deg(v_1)} v_1 [\mathrm{Q}\lambda_{n-1}(v_2, \cdots , v_n)]$$ $$ - \sum_{k,l\geq 2} (-1)^{k+(l-1)(\deg(v_1)+\cdots+\deg(v_k))}[\mathrm{Q}\lambda_k(v_1, \cdots , v_k)] [\mathrm{Q}\lambda_l(v_{k+1},
\cdots, v_n)]. $$
for $n \geq 3$.
Now the operators
$$
 \m_n: \mathbb{H}^*(V)^{\otimes n} \to \mathbb{H}^*(V)
$$
are defined by
$
 \m_n = \mathrm{P} \circ \lambda_n.
$
\begin{theorem}[\cite{Merkulov_SHAKM}]
The operators $\m_n$ define the structure of an $A_{\infty}$-algebra on $\bH^*(V)$
quasi-isomorphic to $(V, \bsd)$.
\end{theorem}

We now compute $\mathrm{Q}$ in our setting.
Since the differential vanishes on our model for $\Hom(S_0, S_1)$,
the operator $\mathrm{Q}$ also vanishes on $\Hom(S_0,S_1)$.
On $\Hom(S_0,S_0)$ and $\Hom(S_1, S_1)$,
we can set
$$
 \mathrm{Q}(z)=x, \ \mathrm{Q}(w)=y, \ \mathrm{Q}(yz)=\frac{xy}{2}, \ \mathrm{Q}(xw)=\frac{xy}{2}
$$
and everything else to be zero.
In the first summand of $\Hom(S_1,S_0)$,
the operator $\mathrm{Q}$ is given by
$$
 \mathrm{Q}(\overrightarrow{z_{1}})= \overrightarrow{x_{1}},
  \
 \mathrm{Q}(\overrightarrow{yz_{1}})= \overrightarrow{xy_{1}}.
$$
A similar formula holds in the second summand.

To compute $\m_3(\overrightarrow{u_1},\overleftarrow{z},\overrightarrow{u_2})$,
we notice that $\overleftarrow{z}\cdot \overrightarrow{u_2}=0$,
so that
\begin{align*}
 \m_3(\overrightarrow{u_1},\overleftarrow{z},\overrightarrow{u_2})
  &= (1-[d,\mathrm{Q}])(\mathrm{Q}(\overrightarrow{u_1}\cdot \overleftarrow{z}) \cdot \overrightarrow{u_2}) \\
  &= (1-[d,\mathrm{Q}])(\mathrm{Q}(z) \cdot \overrightarrow{u_2}) \\
  &= (1-[d,\mathrm{Q}])(x \cdot \overrightarrow{u_2}) \\
  &= (1-[d,\mathrm{Q}])(\overrightarrow{x_2}) \\
  &= \overrightarrow{x_2}.
\end{align*}
The other formulas can be calculated similarly,
and \pref{lm:m_3} is proved.
\end{proof}

We also have the following result:
\begin{lemma}\label{lm:higher_m's}
All $A_\infty$-operations $\m_n$ for $n \geq 4$ vanish.
\end{lemma}
\begin{proof}
We argue using Merkulov's formula by showing that
a higher product cannot have a non-trivial component in any of the cohomology classes.
First we notice that no cohomology class can be written
as a product of two cochains which are in the image of $\mathrm{Q}$.
To avoid repeated arguments,
we will demonstrate why we cannot have
$$
 \m_n(x_1,\cdots,x_n) = \overleftarrow{yz}.
$$
All other cases can be addressed using the same type of arguments.

The only way to write $\overleftarrow{yz}$ as a non-trivial product is as
$y \cdot \overleftarrow{z}$.
Using Merkulov's formula and the above observation,
we can assume without loss of generality that
$\mathrm{Q}\lambda_{n-1}(x_1, \cdots ,x_{n-1})$ has non-trivial coefficient
in the basis vector $y$
and that $x_n$ has a non-trivial component in the basis vector $\overleftarrow{z}$.
This would in turn imply that
$\lambda_{n-1}(x_1, \cdots, x_{n-1})$ has non-trivial coefficient in $w$,
which is not possible unless $n=3$ because $w$ cannot be written as the product of cochains,
$s_1s_2$, where  either $s_1$ or $s_2$ is in the image of $\mathrm{Q}$.
\end{proof}

\pref{lm:m_3} and \pref{lm:higher_m's} show that
the $A_\infty$-operations on $\scF_0$ is identical
to those for the endomorphism algebra of $\mathcal{O}_E \oplus \mathcal{O}_E(-1)$
described in \pref{sc:coh},
and \pref{th:fuk0} is proved.
\end{proof}

For the remainder of this section,
we offer an alternative approach to \pref{th:fuk0},
which stands on a conjecture that we were not able to verify,
but hope is not outside the reach of current technology.
Let
$
 \mathcal{CO}: SH^0(Y^0) \to WF(L,L')
$
be the open closed string map
considered by Abouzaid \cite{Abouzaid_GCGFC}.
This makes the derived category of the wrapped Fukaya category of $Y^0$
into a triangulated category over $SH^0(Y^0)$,
and allows one to talk about the {\em Serre functor} over $SH^0(Y^0)$
in the sense of \cite[Definition 2.5]{Bezrukavnikov-Kaledin_2004}.
A triangulated category over $SH^0(Y^0)$ is {\em Calabi-Yau}
if the Serre functor over $SH^0(Y^0)$ is isomorphic to the shift functor
$\bullet[n]$ for some $n \in \bZ$.


\begin{conjecture} \label{cj:PD}
The morphism $\mathcal{CO}$ makes the derived category
of the wrapped Fukaya category of $Y^0$
into a Calabi-Yau category over $SH^0(Y^0)$.
\end{conjecture}

Let $\scD$ be a triangulated category and
$\scN \subset \scD$ be a full triangulated subcategory.
We will always assume that triangulated categories
have enhancements
in terms of dg categories
\cite{Bondal-Kapranov_ETC}
or $A_\infty$-categories
(cf.~e.g.~\cite{Keller_IAAM}).
The {\em right orthogonal} to $\scN$ is the full subcategory
$\scN^\perp \subset \scD$
consisting of objects $M$ satisfying
$\Hom(N, M) = 0$ for any $N \in \scN$.
The {\em left orthogonal} $\! \,^{\perp} \scN$ is defined similarly.
The subcategory $\scN$ is said to be {\em right admissible}
if the embedding $I : \scN \hookrightarrow \scD$ has
a right adjoint functor $Q : \scD \to \scN$.
Left admissibility is defined similarly
as the existence of a left adjoint functor,
and $\scN$ is said to be {\em admissible}
if it is both right and left admissible.

A subcategory $\scN$ is right admissible
if and only if for any $X \in \scD$,
there exists a distinguished triangle
$
 N \to X \to M \to N[1]
$
with $N \in \scN$ and $M \in \scN^\perp$.
Such a triangle is unique up to isomorphism,
and one has $Q(X) = N$ in this case.
If $\scN$ is right admissible,
then the quotient category $\scD / \scN$ is equivalent to $\scN^\perp$.
Analogous statements also hold for left admissible subcategories.
A sequence $(\scN_1, \ldots, \scN_n)$
of admissible subcategories
in a triangulated category $\scD$ is called
a {\em semiorthogonal decomposition}
$\scN_j \subset \scN_i^\bot$ for any $1 \le j < i \le n$,
and $\scN_1, \ldots, \scN_n$ generates $\scD$
as a triangulated category.
A semiorthogonal decomposition will be denoted by
$$
 \scD = \la \scN_1, \ldots, \scN_n \ra.
$$

An object $E$ of $\scD$ is {\em almost exceptional}
if $\Ext^i(E, E) = 0$ for $i \ne 0$ and
the algebra $A := \Hom(E, E)$ has finite homological dimension
\cite[Definition 2.1]{Bezrukavnikov-Kaledin_2004}.
Let $\scE$ be the smallest full subcategory of $\scD$
containing $E$ and closed under cones and direct summands.
Then one has a semiorthogonal decomposition
$$
 \scD = \la \scE, \ \scE^\bot \ra
$$
as in \cite[Theorem 3.2]{Bondal_RAACS};
the object $N \in \scE$ in the decomposition
$
 N \to X \to M \to N[1]
$
of an object $X \in \scD$ is given by
$$
 N = \hom^\bullet(E, X) \Lotimes_A E,
$$
and $M \in \scE^\bot$ is the mapping cone
$$
 M = \Cone \lb \hom^\bullet(E, X) \Lotimes_A E \xto{\ev} X \rb
$$
of the evaluation morphism.

\begin{proof}[An alternative proof of \pref{th:fuk0}
assuming \pref{cj:PD}]
\pref{cr:ring_isom} and \pref{cr:conifold_tilting}
show that $L_0 \oplus L_1$ is an almost exceptional object
in $D^b \scWtilde$,
so that one has a semiorthogonal decomposition
\begin{align} \label{eq:sod}
 D^b \scWtilde = \la D^b \scW^\bot, \  D^b \scW \ra.
\end{align}
\pref{cj:PD} implies that \eqref{eq:sod}
is an orthogonal decomposition;
\begin{align} \label{eq:od}
 D^b \scWtilde = D^b \scW^\bot \oplus  D^b \scW.
\end{align}
Since $\End(S_0) \cong H^0(S^3) \cong \bC$,
the objects $S_0$ is indecomposable and
belongs to either $D^b \scW$ or $D^b \scW^\bot$.
The latter is impossible since
$\Hom(L_0, S_0) = \bC$.
This implies that $S_0 \in D^b \scW$,
and similarly for $S_1$.
The fact
$$
 \Hom^i(L_j, S_k) =
\begin{cases}
 \bC & i = 0 \text{ and } j = k, \\
 0 & \text{otherwise}
\end{cases}
$$
shows that $S_i$ goes to $\scO_E$ and $\scO_E(-1)$
under the derived equivalence
$$
 \scW \cong D^b \coh X^0,
$$
and \pref{th:fuk0} is proved.
\end{proof}


\section{Immersed Lagrangian $S^2 \times S^1$}
 \label{sc:immersed}

In the construction of the SYZ mirror in \pref{sc:SYZ_construction},
we first considered the fibers away from the discriminant locus
to obtain $\Yv_0$,
and then added fibers above the discriminant locus `by hand'
to obtain a partial compactification $X^0$.
It is reasonable to speculate that
points on $X^0 \setminus \Yv_0$ can be identified
with singular fibers $L_u := \rho^{-1}(u)$ of the original SYZ fibration
$\rho : Y^0 \to \bR^3$,
where $u \in \Gamma$ is a point on the discriminant.
%
In this section,
we give Floer cohomology computations
in favor of this speculation.

Set $a = -1$ and $b = -1/2$ for simplicity,
and consider a point $(1, 0, \lambda) \in \Gamma$
on the discriminant \eqref{eq:SYZ_disc}
for the SYZ fibration \eqref{eq:SYZ_fib}.
The fiber $\Ll := \rho^{-1}(1,0,\lambda)$ is given by
\begin{align*}
 \Ll =
  \lc (z, u_1, v_1, u_2, v_2) \in Y^0 \relmid
   |z| = 1, \
   |u_1| = |v_1|, \
   |u_2|^2 = |v_2|^2 + 2 \lambda
  \rc.
\end{align*}
The Lagrangian $\Ll$ is an $S^1 \times S^1$-fibration
over the unit circle on the $z$-plane shown in \pref{fg:immersed_lag1}
such that the first $S^1$-component degenerates to a point above $z = -1$.
It follows that $\Ll$ is an immersed $S^2 \times S^1$,
where $S^2$ is immersed in the $(z, u_1, v_1)$-direction
in such a way that both the north pole and the south pole
are mapped to $(z, u_1, v_1) = (-1, 0, 0)$,
and $S^1$ is embedded in the $(u_2, v_2)$-direction.
We equip $\Ll$ with the trivial spin structure
and the grading such that
the unique intersection point of $\Ll$ and $L_0$ has Maslov index zero.
We consider the pair
$(\Ll, \nabla_\alpha)$
of $\Ll$ and
a flat $U(1)$ connection $\nabla_\alpha$
on the trivial complex line bundle
$\Ll \times \bC \to \Ll$,
where $\alpha \in U(1)$ is the holonomy
along the generator of $\pi_1(S^2 \times S^1) \cong \bZ$.

\begin{figure}[ht]
\centering
\input{immersed_lag1.pst}
\caption{The immersed Lagrangian $\Ll$}
\label{fg:immersed_lag1}
\end{figure}

Immersed Lagrangian Floer theory
\cite{MR2155230,Akaho-Joyce}
gives a structure of an $A_\infty$-algebra
on $H^*((\Ll, \nabla_\alpha); \bC)$.
The following lemma is a corollary of a result of Abouzaid,
which can be found in \cite[Lemma 11.6]{Seidel_CDST}:

\begin{lemma} \label{lm:Abouzaid}
The $A_\infty$-algebra on
$H^*((\Ll, \nabla_\alpha); \bC)$
is quasi-isomorphic
to the exterior algebra $\Lambda^*(\bC^3)$
equipped with the trivial differential.
\end{lemma}

\begin{proof}
The immersed Lagrangian $\Ll$ is exact in the sense that
the symplectic form $\omega$ vanishes on $\pi_2(Y^0, \Ll)$,
since it is homotopic to $S_0 \cup S_1$ appearing in \pref{th:fuk0}.
It follows that
the $A_\infty$-structure on $H^*((\Ll, \nabla_\alpha); \bC)$
does not depend on $\nabla_\alpha$
and can be computed by the chain model of Abouzaid \cite{MR2786590}.

The Abouzaid model for $\Ll$ is the tensor product
of the Abouzaid model for an immersed Lagrangian $S^2$
and the cochain complex $C^*(S^1)$ for a circle.
Since the Abouzaid model for the immersed $S^2$
is quasi-isomorphic to $\Lambda^*(\bC^2)$ by \cite[Lemma 11.6]{Seidel_CDST}
and the cochain complex $C^*(S^1)$ is quasi-isomorphic to $\Lambda(\bC)$,
their tensor product is quasi-isomorphic to
$
 \Lambda(\bC^2) \otimes \Lambda(\bC) \cong \Lambda^*(\bC^3).
$
\end{proof}

Since the immersed Lagrangian $\Ll$ is not exact
in the sense that the pull-back of the Liouville 1-form
(i.e., the 1-form $\theta$ on $Y^0$ such that $\omega = d \theta$)
represents a non-trivial cohomology class on $H^1(S^2 \times S^1)$,
one has to work over the Novikov field
$$
 \Lambda_\bC = \lc \sum_{i=0}^\infty a_i T^{\lambda_i} \relmid a_i \in \bC, \
 \lim_{i \to \infty} \lambda_i = \infty \rc
$$
to define the Floer cohomology $\Hom(L_i,(\Ll, \nabla_\alpha))$
for $i=0,1$.
We replace the mirror manifold $X^0$
with the variety $\Xl := X^0 \otimes_{\bC} \Lambda$
over $\Lambda$ accordingly.
Let $\pl$ be the point on $\Xl$
given by $(u,v,w_1,w_2,[x:y]) = (0,0,-1,\alpha T^{\lambda},[0:1])$.

\begin{lemma} \label{lm:point}
The Floer cohomology
$
 \Hom(L_0, (\Ll, \nabla_\alpha))
  \oplus
 \Hom(L_1, (\Ll, \nabla_\alpha))
$
as a module over $\oplus_{i,j=0}^1\Hom(L_i,L_j)$
is isomorphic
to the module
$
 \Hom(\scO_\Xl \oplus \scO_\Xl(1), \scO_\pl)
$
over $\End(\scO_\Xl \oplus \scO_\Xl(1))$.
\end{lemma}

\begin{proof}
The intersection $L_0 \cap \Ll$ consists of a single point
$
 \ql = (1, \sqrt{2}, \sqrt{2}, u_2, v_2) \in Y^0,
$
where $(u_2, v_2) \in (\bR^{>0})^2$ is defined by
$
 u_2 v_2 = 3/2
$
and
$
 u_2^2 - v_2^2 = 2 \lambda.
$
The $A_\infty$-operation
$
 \frakm_2(\ql, p_{a,i,j})
$
in immersed Lagrangian Floer theory
is given by the virtual count
$$
 \frakm_2(\ql, p_{a,i,j}) = \sum_{\phi_n(q) \in \Ll \cap \phi_n(L_0)}
  \sum_{\varphi \in [\scMbar(q)]^{\mathrm{virt}}} \sgn(\varphi)
   \hol(\nabla_\alpha, \varphi(\partial D^2))
   T^{\int_D^2 \varphi^* \omega}
   \cdot q
$$
over the moduli space $\scMbar(q)$ of holomorphic maps
$\varphi : (D^2, (z_0, z_1, z_2)) \to Y_0$
from a disk with three marked points on the boundary
satisfying
\begin{itemize}
 \item
$\varphi(z_0) = \phi_n(q)$,
$\varphi(z_1) = p_{a,i_1,i_2}$, and
$\varphi(z_2) = \ql$,
 \item
$\varphi(\partial_0 D^2) \subset \phi_n(L_0)$,
$\varphi(\partial_1 D^2) \subset L_0$, and
$\varphi(\partial_2 D^2) \subset \Ll$.
\end{itemize}
Here $\partial_i D^2 \subset \partial D^2$ is the interval
between $z_i$ and $z_{i+1}$.
The sign $\sgn(\varphi)=\pm 1$ comes from the orientation on the moduli space
(cf. \cite[Section 11]{Seidel_PL} and \cite[Chapter 8]{Fukaya-Oh-Ohta-Ono}).

\begin{figure}[ht]
\begin{minipage}{.5 \linewidth}
\centering
\vspace{1mm}
\input{immersed_intersection.pst}
\vspace{3mm}
\caption{Intersections on the base}
\label{fg:immersed_intersection}
\end{minipage}
\begin{minipage}{.5 \linewidth}
\centering
\input{fiber_intersection.pst}
\caption{Intersections on the fiber}
\label{fg:fiber_intersection}
\end{minipage}
\end{figure}

\begin{figure}[ht]
\begin{minipage}{.5 \linewidth}
\centering
\vspace{4mm}
\input{base_triangle1.pst}
\caption{A triangle on the base}
\label{fg:base_triangle1}
\end{minipage}
\begin{minipage}{.5 \linewidth}
\centering
\input{base_triangle2.pst}
\caption{Another triangle on the base}
\label{fg:base_triangle2}
\end{minipage}
\end{figure}

In immersed Lagrangian Floer theory,
one counts only maps $\varphi$ such that
$\varphi|_{\partial_2 D^2} : \partial_2 D^2 \to \Ll$
lifts to a map $\partial_2 D^2 \to S^2 \times S^1$.
One can see from \pref{fg:immersed_intersection}
that this is possible only for $a=0$, so that
\begin{align}
 \frakm_2(\ql, p_{a,i,j}) &= 0
  \label{eq:immersed_A-infinity1}
\end{align}
for $a \ne 0$ and any $i, j \in \bZ$.
One can also see from Figures \ref{fg:immersed_intersection}
and \ref{fg:fiber_intersection}
that for each $(i, j) \in \bZ^2$,
there is a unique holomorphic triangle
bounded by $\phi_n(L_0)$,
$L_0$ and $\Ll$
two of whose vertices are $p_{0,i,j}$
and $q_\lambda$.
The projection of this unique triangle to the $z$-plane is shown
in \pref{fg:base_triangle1}.
The third vertex of this triangle is $\phi_n(q_\lambda)$,
which is the unique intersection point
of $\phi_n(L_0)$ and $\Ll$.
Since $L_0$ and hence $\phi_n(L_0)$ carry
the trivial spin structure and the trivial flat $U(1)$-bundle,
the factor
$
 \sgn(\varphi)
   \hol(\nabla_\alpha, \varphi(\partial D^2))
$
comes entirely from the holonomy and the spin structure
along $\partial_2 D^2$.
Since $\varphi(\partial_2 D^2)$ wraps $j$ times
along the $S^1$-factor of $S^2 \times S^1$,
one has
$
 \hol(\nabla_\alpha, \varphi(\partial D^2)) = \alpha^j.
$
The trivial spin structure on $S^2 \times S^1$ induces
the trivial spin structure on the $S^1$-factor,
and the non-trivial spin structure on the equator
of the $S^2$-factor
(cf. e.g. \cite[Page 201]{MR0157388}).
This non-trivial spin structure contributes
to the sign as
$
 \sgn(\varphi) = (-1)^i,
$
cf. \cite[Section 9e]{Seidel_K3}.
This sign can also be determined
by the ring isomorphism
$\Hom_{\scW_1}(L_0, L_0) \cong H^0 \lb \scO_{X^0} \rb$
and the associativity.
The area of $\varphi(D^2)$ is $j \lambda$
up to an additive overall constant which can be absorbed
in the definition of the generator of the Floer cohomology,
and one obtains
\begin{align*}
 \frakm_2(\ql, p_{0,i,j}) &= (-1)^i (\alpha_2 T^\lambda)^j \ql
\end{align*}
for any $i, j \in \bZ$.

The intersection $L_1 \cap \Ll$ also consists of a single point $\rl$,
and one can show
$
 \frakm_2(\rl, p_{0,i,j}') = (-1)^i (\alpha_2 T^\lambda)^j \rl
$
for any $i, j \in \bZ$
by exactly the same argument as above,
where $p_{a,i,j}' \in \phi_n(L_1) \cap L_1$ are defined
similarly as $p_{a,i,j} \in \phi_n(L_0) \cap L_0$.
One can also see from \pref{fg:base_triangle2}
and the same argument as above
that the composition
$
 \Hom(L_1, \Ll) \otimes \Hom(L_0, L_1) \to \Hom(L_0, \Ll)
$
is given by
$$
 \frakm_2(\rl, q_{a,i_1,i_2}) =
\begin{cases}
 (-1)^i (\alpha_2 T^\lambda)^j \ql, & a = 1/2, \\
 0 & \text{otherwise}.
\end{cases}
$$
The composition
$
 \Hom(L_0, \Ll) \otimes \Hom(L_1, L_0) \to \Hom(L_1, \Ll)
$
can be computed similarly,
and \pref{lm:point} is proved.
\end{proof}

\section{Small toric Calabi-Yau 3-folds}
 \label{sc:small}

Let $Y^0$ be the complete intersection in
$
 \bCx \times \bC^4
  = \Spec \bC[z, z^{-1}, u_1, u_2, v_1, v_2]
$
defined by
\begin{align} \label{eq:Y}
\left\{
\begin{aligned}
  u_1 v_1 &= (z-a_1)\cdots(z-a_k), \\
 u_2 v_2 &= (z-b_1)\cdots(z-b_l). \\
\end{aligned} \right.
\end{align}
The SYZ mirror for $Y^0$ is the complement
$$
 X^0 = X \setminus D
$$
of a divisor $D$ in a crepant resolution $X$ of the toric variety
whose fan polytope is shown in  \pref{fg:general_diagram}.

\begin{figure}[ht]
\centering
\input{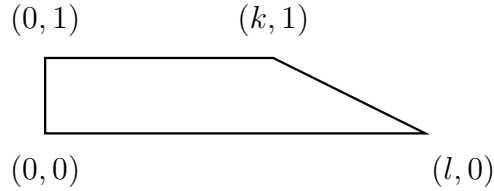}
\caption{The fan polytope}
\label{fg:general_diagram}
\end{figure}

Here, the fan polytope of a toric variety
is the convex hull of the primitive generators
of one-dimensional cones of the fan.
The fan structure induces a polyhedral decomposition
of the fan polytope,
and the fan polytope equipped with this polyhedral decomposition
is called a {\em toric diagram}.

The construction of the SYZ mirror of a complete intersection
in \cite[Section 11]{Abouzaid-Auroux-Katzarkov_LFB}
shows that primitive generators
of one-dimensional cones of the fan for $X$
are given by
$(0,1,0), (1, 1, 0), \ldots, (k,1,0)$, and
$(0,0,1), (1, 0, 1), \ldots, (l, 0, 1)$.

One can map these points
by the automorphism of $N \cong \bZ^3$
sending $(n_1, n_2, n_3)$ to $(n_1, n_2, n_2 + n_3)$,
so that the fan polytope is the quadrangle
on the hyperplane $\{ (n_1, n_2, n_3) \in N_\bR \mid n_3 = 1 \}$
shown in \pref{fg:general_diagram}.
The toric Calabi-Yau 3-fold $X$ is {\em small}
in the sense that the resolution $X \to Z = \Spec \bC[X]$
does not have 2-dimensional fibers
(in other words, the toric variety $X$ has no compact toric divisors).

It is sometimes convenient
to consider a stacky resolution $\scX$ of $Z$,
whose toric diagram is obtained by the triangulation
of the fan polytope.
Let consider the case
when the fan for $\scX$ has two 3-dimensional cones,
one of which is generated by
\begin{align*}
 v_1 = (0, 0, 1), \
 v_3 = (0, 1, 0), \text{ and }
 v_4 = (l, 0, 1),
\end{align*}
and the other is generated by
\begin{align*}
 v_2 = (k, 1, 0), \ v_3, \text{ and } v_4.
\end{align*}
Let $\varphi : \bZ^4 \to N \cong \bZ^3$
be the homomorphism
sending the $i$-th standard basis $e_i \in \bZ^4$
to $v_i \in N$ for $i = 1, \ldots, 4$.
Then the toric stack $\scX$ is the quotient
$$
 \scX := [(\bC^4 \setminus \Sigma)/K],
$$
where
$
 \Sigma := \lc (x_1, x_2, x_3, x_4) \in \bC^4 \relmid x_1 = x_2 = 0 \rc
$
is the Stanley-Reisner locus and
\begin{align*}
 K
 &= \Ker \lb \varphi \otimes \bCx : (\bCx)^4 \to N_{\bCx} \cong (\bCx)^3 \rb \\
 &= \lc (\alpha_1, \alpha_2, \alpha_3, \alpha_4) \in (\bCx)^4
  \relmid \alpha_2^k \alpha_4^l = \alpha_2 \alpha_3 = \alpha_1 \alpha_4 = 1 \rc \\
 &\cong K_1 \times K_2.
\end{align*}
Here $K_1$ and $K_2$ are the subgroups of $K$ given by
\begin{align}
 K_1 &:= \lc ( \alpha^k, \alpha^l, \alpha^{-l}, \alpha^{-k}) \in (\bCx)^4
   \relmid \alpha \in \bCx \rc
  \cong \bCx, \\
 K_2 &:= \lc ( \alpha, 1, 1, \alpha^{-1}) \in (\bCx)^4
   \relmid \alpha \in \bCx, \ \alpha^g = 1 \rc
  \cong \bZ / g \bZ
\end{align}
where $g = \gcd(k, l)$.
This shows that the toric stack $\scX$
is the total space of the direct sum of two line bundles
on the quotient stack $\bX := \bX_{k,l} := [\bP(k', l') / (\bZ / g \bZ)]$
of the weighted projective line
$\bP(k', l')$
for $k = g k'$ and $l = g l'$.

The toric stack $\bX$ has the following description
due to Geigle and Lenzing \cite{Geigle-Lenzing_WPC}:
Let $S = \bC[x_1, x_2]$ be the polynomial ring in two variables,
graded by the abelian group
$
 L = \bZ \cdot \vecx_1 \oplus \bZ \cdot \vecx_2 / (k \vecx_1 - l \vecx_2)
$
of rank one as $\deg (x_i) = \vecx_i$ for $i = 1, 2$.
Then $L$ can naturally be identified
with the group $\Hom(K, \bCx)$ of characters of $K$
and one has
$$
 \bX \cong [(\Spec S \setminus \bszero) / K].
$$
The Picard group of $\bX$ can be identified with $L$,
and the line bundle on $\bX$ associated with an element $\vecx \in L$
will be denoted by $\scO_\bX(\vecx)$.
The canonical bundle of $\bX$ is $\scO_\bX(\vecomega)$
for $\vecomega = - \vecx_1 - \vecx_2$,
and the stack $\scX$ is the total space of the direct sum
$$
 \scX \cong \scO_\bX(- \vecx_1) \oplus \scO_\bX(- \vecx_2)
$$
of line bundles.

Choose $a_i$ and $b_j$ in such a way
that all of them are on the unit circle and mutually distinct.
Let $(\gamma_i)_{i=0}^{k+l-1}$ be a collection
of strongly admissible paths,
such that for any connected component of $S^1 \setminus \Delta$
for $S^1 = \{ z \in \bCx \mid |z| = 1 \}$
and $\Delta = \{ a_1, \ldots, a_k, b_1, \ldots, b_l \}$,
there is a unique $i$ such that $\gamma_i$ intersects it.
Let $\scW$ be the wrapped Fukaya category of $Y$
consisting of $L_i := L_{\gamma_i}$ for $i = 0, \ldots, k+l-1$.
Define the collection $(\scL_i)_{i=0}^{k+l-1}$
of line bundles on $\scX$ inductively by $\scL_0 = \scO_\scX$ and
\begin{align*}
 \scL_i =
\begin{cases}
 \scL_i \otimes \pi^*(\scO(\vecx_1)) &
  \text{$a_j$ for some $j$ lies between $\gamma_{i-1}$ and $\gamma_i$}, \\
 \scL_i \otimes \pi^*(\scO(\vecx_2)) &
  \text{$b_j$ for some $j$ lies between $\gamma_{i-1}$ and $\gamma_i$},
\end{cases}
\end{align*}
where $\pi : \mathcal{X} \to \bX$ is the natural projection. \vskip 10 pt

Then the proof of \pref{th:hms}
can be easily adapted to prove the following:

\begin{theorem} \label{th:hms2}
There is an equivalence
\begin{align}
 D^b \scW \cong D^b \coh \mathcal{X}^0
\end{align}
of triangulated categories
sending $L_i$ to $\scL_i$ for $i = 0, \ldots, k+l-1$.
\end{theorem}

The manifold $Y^0$ comes in a family
$\scY^0 \to S$
over the configuration space
$$
  S = \left. \lc (a_1, \ldots, b_l) \in (\bCx)^{k+l} \relmid
  \text{all the points $a_1, \ldots, b_l$ are distinct} \rc
  \right/ \mathfrak{S}_k \times \mathfrak{S}_l,
$$
in such a way that $\scY^0$ is the submanifold of $\bCx \times \bC^4 \times S$
defined by the same equations \eqref{eq:Y} as $Y^0$.
This family is locally trivial as a family of symplectic manifolds
by Moser's theorem.
The fundamental group $\MA_{k,l} := \pi_1(S)$ is
called the {\em mixed annular braid group}, and
the symplectic monodromy gives a homomorphism
$$
 \phi : \MA_{k,l} \to \pi_0(\Symp(Y^0, \omega))
$$
to $\pi_0$ of the symplectomorphism group of $Y^0$.

Choose the point
$$
 (a_1, a_2, \ldots, b_l)
 = (\zeta_{k+l}, \zeta_{k+l}^2, \ldots, \zeta_{k+l}^{k+l}) \in S,
  \quad
 \zeta_{k+l}=\exp(2 \pi \sqrt{-1}/(k+l))
$$
as a base point, and let $\gamma_i$
be the line segment
from $\zeta_{k+l}^i$ to $\zeta_{k+l}^{i+1}$.
The half-twist $T_i$ along $\gamma_i$
interchanges $\zeta_{k+l}^i$ and $\zeta_{k+l}^{i+1}$, and
one can see that $T_i$ for $i \ne k, k+l$
and $(T_i)^2$ for $i=k, k+l$
belong to $\MA_{k,l}$.
Let $S_i$ be the compact Lagrangian submanifold of $Y^0$
defined by the path $\gamma_i$
as in \eqref{eq:compact_lag}.
The Lagrangian $S_i$ is homeomorphic to $S^3$
if $i=k, k+l$ and $S^2 \times S^1$ otherwise.
\begin{itemize}
 \item
For $i \neq k, k+l$,
we can identify a neighborhood of $S_i$ with $T^*S^1 \times T^*S^2$, and
the symplectic monodromy along $T_i$
is given by the symplectic Dehn twist in the $T^*S^2$ factor.
 \item
For $i=k, k+l$,
the symplectic monodromy along $(T_i)^2$
is the Dehn twist along the Lagrangian $S_i$.
\end{itemize}
It is an interesting problem
to explore the relation between this action
and the mixed braid group action
on the derived category of coherent sheaves on $X^0$
studied by Donovan and Segal
\cite{Donovan-Segal_MBG}.

\appendix

\section{Lefschetz wrapped Floer cohomology}

\subsection{Liouville domains and wrapped Floer cohomology}

An exact symplectic manifold with contact type boundary,
or a {\em Liouville domain} for short,
is a pair $(M, \theta)$ of a compact manifold $M$ with boundary
and a one-form $\theta$ on $M$ called the {\em Liouville one-form}
such that
\begin{itemize}
 \item
the two-form $\omega := d \theta$ is a symplectic form on $M$,
and
 \item
the {\em Liouville vector field} $Z$
defined by $\iota_Z \omega = \theta$
points strictly outward along the boundary $\partial M$.
\end{itemize}
The restriction $\alpha := \theta|_{\partial M}$
of the Liouville one-form
is a contact one-form on $\partial M$.
The {\em Reeb vector field} $R$ on $\partial M$
is defined by $R \in \Ker \alpha$ and $\alpha(R) = 1$.
The {\em symplectic completion} $\hat{M}$ of $M$ is obtained
by gluing the positive part
$$
 \lb \partial M \times [1,\infty), d(r \alpha) \rb
$$
of the symplectization of $\partial M$ onto $M$;
$$
 \hat{M} := M \cup_{\partial M} \lb \partial M \times [1,\infty) \rb.
$$
Let $L$ be a compact Lagrangian submanifold $L$ of $M$
such that
\begin{itemize}
 \item
$\theta |_L$ is exact; $\theta|_L = d f$,
 \item
$L$ intersects $\partial M$ transversally, and
 \item
$\theta|_L$ vanishes to infinite order
along the boundary $\partial L := L \cap \partial M$.
\end{itemize}
In this setting,
the completion
$$
 \hat{L} := L \cup_{\partial L} \lb \partial L \times [1, \infty) \rb,
$$
of $L$ is a Lagrangian submanifold of $\hat{M}$.

A Hamiltonian function $H \in C^\infty(\hat{M})$ is {\em admissible}
if there are constants $K > 0$, $a > 0$, and $b$ such that
\begin{align} \label{eq:H_rho}
 H(x, r) = a r + b,
 \qquad \forall (x, r) \in \partial M \times [K, \infty) \subset \hat{M}.
\end{align}
The constant $a$ is called the {\em slope} of $H$.
An almost complex structure $J$ on $\hat{M}$ is {\em admissible}
if it is of contact type
outside a compact set;
\begin{align} \label{eq:J_contact}
 d r = \theta \circ J,
  \qquad \forall (x, r) \in \partial M \times [K, \infty) \subset \hat{M}.
\end{align}
A {\em Reeb chord} of length $w$ is a trajectory
$x : [0,w] \to \partial M$ of the flow along $R$
such that $x(0) \in L$ and $x(w) \in L$.
An {\em integer Reeb chord} is a Reeb chord of integer length.
A {\em Hamiltonian chord} is defined similarly
as a trajectory of the Hamiltonian vector field
starting and ending at $L$.
If we write the time $t$ Hamiltonian flow as
$\varphi_t : \hat{M} \to \hat{M}$,
then a Hamiltonian chord of length $w$ corresponds
to an intersection point $p \in L \cap \varphi_w(L)$.
A Hamiltonian chord is {\em non-degenerate}
if the corresponding intersection is transversal.

Fix an admissible Hamiltonian $H$ of unit slope.
If $\dim M \ge 4$,
then by perturbing $L$ by an exact symplectic isotopy if necessary,
we may assume
\cite[Lemmas 8.1 and 8.2]{Abouzaid-Seidel_OSAVF}
that
\begin{itemize}
 \item
there are no integer Reeb chords,
 \item
all integer Hamiltonian chords are non-degenerate, and
 \item
no point of $L$ is both a starting point of an integer Hamiltonian chord
and an endpoint of an integer Hamiltonian chord,
which may or may not be the same chord.
\end{itemize}
For an integer $w$,
the set of Hamiltonian chords of length $w$ will be denoted by $\scX_w$.
The set $\scX_w$ is finite
since all the integer Hamiltonian chords are non-degenerate.
The {\em action} of $x \in \scX_w$ is defined by
$$
 A_{w H}(x) = \int_0^1 \big( x^* \theta - w H(x(t)) d t \big)
  + H(x(1)) - H(x(0)).
$$
The Floer complex is defined as the direct sum
\begin{align*}
 \CF^*(\hat{L}; wH) := \bigoplus_{x \in \scX_w} \bC [x],
\end{align*}
equipped with the grading
coming from the Maslov index.
The differential $\delta$ on $\CF(\hat{L}; wH)$ is given
by counting solutions to Floer's equation
\begin{align*}
\begin{cases}
 u : \bR \times [0,1] \to \hat{M}, \\
 u( \bR \times \{ 0, 1 \} ) \subset \hat{L}, \\
 \lim_{s \to \pm \infty} u(s, \cdot) = x_{\pm}(\cdot), \\
 \partial_s u + J_t(\partial_t u - w X_H) = 0
\end{cases}
\end{align*}
up to the $\bR$-action by translation in the $s$-direction.
Here, one has to choose a $t$-dependent almost complex structure
to achieve transversality in Floer's equation.
The conditions
\eqref{eq:H_rho} and \eqref{eq:J_contact}
gives the maximum principle for $u$,
which ensures the compactness of the moduli space.
The continuation map
$$
 \varphi : \CF^*(\hat{L}; wH) \to \CF^*(\hat{L}; (w+1)H)
$$
is defined as the sum
$$
 \varphi(x_+) = \sum_{u \in \scM(x_-, x_+)} \pm x_-
$$
over the set $\scM(x_-, x_+)$ of solutions to the continuation equation
\begin{align*}
\begin{cases}
 u : \bR \times [0,1] \to \hat{M}, \\
 u( \bR \times \{ 0, 1 \} ) \subset \hat{L}, \\
 \lim_{s \to \pm \infty} u(s, \cdot) = x_{\pm}(\cdot), \\
 \partial_s u + J_{s,t}(\partial_t u - X_s) = 0
\end{cases}
\end{align*}
where $J_{s,t}$ is an $(s, t)$-dependent almost complex structure
and $X_s$ is the Hamiltonian vector field
of an $s$-dependent Hamiltonian $H_s$
which coincides with $(w+1) H$ and $w H$
when $s \ll 0$ and $s \gg 0$ respectively.
A standard argument in Floer theory shows that the continuation map is a chain map,
which is independent of the choice of almost complex structures
up to chain homotopy.
The {\em wrapped Floer cohomology} is defined by
$$
 \HW(\hat{L}) := \varinjlim_w \HF(\hat{L}; wH).
$$
The continuation map is defined more generally
for a family $H_s$ of admissible Hamiltonians
with monotonically decreasing slope,
and satisfies the transitive law;
if one divides a family
$\{ H_s \}_s$ from $H_{-\infty}$ to $H_\infty$
smoothly into two,
then the diagram
\begin{align*}
\begin{psmatrix}
 \HF(\hat{L}; H_\infty) & & \HF(\hat{L}; H_{-\infty}) \\
  & \HF(\hat{L}; H_0)
\end{psmatrix}
\psset{nodesep=5pt,arrows=->,shortput=nab}
\ncline{1,1}{1,3}^{\varphi}
\ncline{1,1}{2,2}_{\varphi_+}
\ncline{2,2}{1,3}_{\varphi_-}
\end{align*}
consisting of continuation maps commutes.
We say that a family $\{ H_p \}_p$ of admissible Hamiltonians is {\em cofinal}
if the slope of $H_p$ goes to infinity as $p$ goes to infinity.
The wrapped Floer cohomology can also be defined
as the limit of Floer cohomologies
with respect to any cofinal family
of non-degenerate admissible Hamiltonians.

The {\em triangle product} on wrapped Floer cohomology
is defined by counting solutions
of the inhomogeneous Cauchy-Riemann equation
\begin{align*}
\begin{cases}
 u : S \to \hat{M}, \\
 u( \partial S ) \subset \hat{L}, \\
 \lim_{s \to \pm \infty} u( \epsilon^k(s, \cdot)) = x^k(\cdot),
 \qquad k = 0, 1, 2, \\
 (d u_z - X_{u(z)} \otimes \gamma_z) \circ j
   + J_{z, u(z)} \circ (d u_z - X_{u(z)} \otimes \gamma_z) = 0,
\end{cases}
\end{align*}
where
\begin{itemize}
 \item
$w^k \in \bN$ and $x^k \in \scX_{w^k}$ for $k = 0, 1, 2$,
 \item
$S = D^2 \setminus \{ \zeta^0, \zeta^1, \zeta^2 \}$ is a disk
with three points on the boundary removed,
 \item
$\epsilon^0 : (-\infty, 0] \times [0,1] \to S$
and $\epsilon^{1,2} : [0, \infty) \times [0,1] \to S$
are {\em strip-like ends},
 \item
$j$ is the complex structure on $S$,
 \item
$\{ J_z \}_{z \in S}$ is a family of almost complex structures on $\hat{M}$,
 \item
$\gamma$ is a one-form on $S$
satisfying
\begin{itemize}
 \item
$\gamma|_{\partial S} = 0$,
 \item
$d \gamma < 0$ on $S$,
 \item
$d \gamma = 0$ in a neighborhood of $\partial S$,
 \item
$(\epsilon^k)^* \gamma = w^k d t$ on the strip-like ends, and
\end{itemize}
 \item
$X \otimes \gamma \in \Hom(TS, u^*T\hat{M})$ is obtained
by composing $\gamma \in C^\infty(T^* S)$
with $u^*X \in C^\infty(u^* TM)$.
\end{itemize}
A more careful discussion
on the wrapped Floer cohomology can be found in \cite{Ritter_TQFTSC}.
To define higher $A_\infty$-operations,
one takes the homotopy colimit of the Floer cochain complex
instead of the colimit of the cohomology,
and use moduli spaces of {\em stable popsicle maps}
\cite{Abouzaid-Seidel_OSAVF}.

\subsection{Lefschetz fibrations and wrapped Floer cohomology}\label{sc:lefschetz_WFC}

This section follows McLean \cite{MR2497314} closely.
An {\em exact Lefschetz fibration} is a proper map
$
 \pi : E \to S
$
from a compact manifold $E$ with corners
to a compact surface $S$ with boundary
satisfying the following:
\begin{itemize}
 \item
$\partial E$ consists of
the {\em vertical boundary} $\partial_v E := \pi^{-1}(\partial S)$ and
the {\em horizontal boundary} $\partial_h E := \partial E \setminus \partial_v E$
meeting in a codimension 2 corner.
 \item
$\pi$ is a $C^\infty$-map with finitely many critical points
$E^\crit \subset E$ with critical values $S^\crit \subset S$.
Every critical point is non-degenerate
in the sense that the Hessian is non-degenerate.
Different critical points have distinct critical values.
 \item
$E$ is equipped with a one-form $\Theta$
such that $\Omega = d \Theta$ is a symplectic form on $E_s \setminus E^\crit$
for every $s \in S$,
where $E_s := \pi^{-1}(s)$ is the fiber of $\pi$.
 \item
There is a neighborhood $N$ of $\partial_h E$
such that $\pi|_N : N \to S$ is a product fibration $S \times U$,
where $U$ is a neighborhood of $\partial F$ in $F$.
We require that $\Theta|_N$ is a pull-back from the second factor.
 \item
There are integrable complex structures $J_0$ (resp. $j_0$)
defined on a neighborhood of $E^\crit$ (resp. $S^\crit$)
such that $\pi$ is $(J_0, j_0)$-holomorphic near $E^\crit$.
 \item
$\Omega$ is a K\"{a}hler form for $J_0$ near $E^{\crit}$.
\end{itemize}


There is a natural connection for $\pi$
given by the horizontal distribution
defined as the $\Omega$-orthogonal to the tangent space
to the fiber.
Parallel transport with respect to this connection
gives exact symplectomorphisms between
smooth fibers of $\pi$.
We write a smooth fiber of $\pi$
considered as an abstract exact symplectic manifold
as $F$.

We say that $E$ is a {\em compact convex Lefschetz fibration}
if $(F, \Theta|_F)$ is a Liouville domain.
Choose a Liouville one-form $\theta_S$ on the base $S$.
Then there is a constant $K > 0$ such that for all $k \ge K$,
one has
\begin{itemize}
 \item
$\omega := \Omega + k \pi^* \omega_S$ is a symplectic form
 \item
the $\omega$-dual $\lambda$ of $\theta := \Theta + k \pi^* \theta_S$
is transverse to $\partial E$
and pointing outward
\end{itemize}
by \cite[Theorem 2.15]{MR2497314}.
One can complete a compact convex Lefschetz fibration
to a {\em complete convex Lefschetz fibration} $\pi : \hat{E} \to \hat{S}$ in a natural way,
whose base is the completion $\hat{S}$ of the base $S$
and whose fiber is a completion $\hat{F}$ of the fiber $F$.
The completion $\hat{E}$ can be partitioned into
\begin{itemize}
 \item
$E \subset \hat{E}$,
 \item
$A := F_e \times \hat{S}$
where $F_e$ is the cylindrical end of $\hat{F}$, and
 \item
$B := \hat{E} \setminus (A \cup E)$
\end{itemize}
as in \cite[Figure 1]{MR2497314}.

The completion $\hat{E}$ is isomorphic
to the completion $\hat{M}$ of a Liouville domain $M$,
obtained by smoothing out the corner of $E$.
We write the radial coordinates for cylindrical ends of $\hat{E}$, $\hat{S}$ and $\hat{F}$
as $r$, $r_S$ and $r_F$.
There exists a positive constant $\varpi$
such that $r_S \le \varpi r$ and $r_F \le \varpi r$
by \cite[Lemma 5.7]{MR2497314}.

A map $H : \hat{E} \to \bR$ is a {\em Lefschetz admissible Hamiltonian}
if
$
 H|_A = \pi^*H_S + \pi_1^*H_{F}
$
outside some large compact set
\cite[Definition 2.21]{MR2497314}.
Here $H_S$ and $H_F$ are admissible Hamiltonians
on $\hat{S}$ and $\hat{F}$
such that $H_S = 0$ on $S \subset \hat{S}$
and $H_F = 0$ on $F \subset \hat{F}$ respectively
\cite[Page 1905]{MR2497314}, and
$\pi_1 : A = F_e \times \hat{S} \to F_e$ is the first projection.

Let $\gamma : [0, 1] \to S$ be a path on the base
such that $\gamma((0,1)) \subset S \setminus S^\crit$.
Recall that a {\em Lagrangian submanifold fibered over $\gamma$}
is a Lagrangian submanifold $L$ of $E$
obtained as the trajectory of the parallel transport along $\gamma$
of a Lagrangian submanifold $L_s$ in a fiber $E_s = \pi^{-1}(s)$.
We assume that
$L$ is exact,
$L$ intersects
$\partial E$ transversally,
and
$\theta|_L$ vanishes to infinite order along $\partial L$.
If an endpoint of $\gamma$
is in the interior of $S$,
then it must be a critical value of $\pi$.
If exactly one endpoint of $\gamma$ is in the interior of $S$,
then $L$ is a Lefschetz thimble.
If both endpoints of $\gamma$ are in the interior of $S$,
then $L$ is a Lagrangian sphere.
The Lagrangian $L \subset M$ can be completed
to a Lagrangian $\hat{L} \subset \hat{M}$
by first taking the completion
$
 \hat{L}_s := L_s \cup_{\partial L_s} ([1,\infty) \times \partial L_s)
  \subset \hat{E_s}$
in the fiber direction
and then taking its parallel transport along
$
 \hat{\gamma} = \gamma \cup_{\partial \gamma} ([1, \infty) \times \partial \gamma)
  \subset \hat{S}.
$
Since $\hat{L} \cap A$ is the product $(\hat{L}_s \setminus L_s) \times \hat{\gamma}$
and $\hat{L} \cap B$ is the product $L_s \times (\hat{\gamma} \setminus \gamma)$,
one has a maximum principle which applies to Lefschetz admissible $H$:

\begin{lemma}
For any Floer trajectory $u : D \to \hat{E}$,
the functions $r_S \circ u$ and $r_F \circ u$ do not admit local maxima
for large $r_S$ and $r_F$.
\end{lemma}

This allows one to define the Floer differential
and the continuation map,
which gives the {\em Lefschetz wrapped Floer cohomology}
$$
 \HW_l^*(\hat{L}) := \varinjlim_w \HF^*(\hat{L}; w H).
$$
The Lefschetz wrapped Floer cohomology $\HW_l^*(\hat{L})$ does not depend
on the choice of a Lefschetz admissible Hamiltonian
just as in the case of the ordinary wrapped Floer cohomology.


\begin{theorem} \label{th:appendix}
One has an isomorphism
\begin{align} \label{eq:appendix}
 \HW^*(\hat{L}) \cong \HW_l^*(\hat{L})
\end{align}
of graded rings.
\end{theorem}

The isomorphism \eqref{eq:appendix}
is obtained by
\begin{align}
 \HW^*(\hat{L})
  &\cong \varinjlim_p \HF^*(\hat{L}; \varrho_p)
   \label{eq:isom1} \\
  &\cong \varinjlim_p \HF^*_{[-\epsilon,\infty)}(\hat{L}; \varrho_p)
   \label{eq:isom6} \\
  &\cong \varinjlim_p \HF^*_{[-\epsilon,\infty)}(\hat{L}; K_p)
   \label{eq:isom2} \\
  &\cong \varinjlim_p \HF^*_{[-\epsilon, \infty)}(\hat{L}; G_p)
   \label{eq:isom3} \\
  &\cong \varinjlim_p \HF^*(\hat{L}; G_p)
   \label{eq:isom4} \\
  &\cong \HW_l^*(\hat{L}),
   \label{eq:isom5}
\end{align}
which is an adaptation of the proof of
\cite[Theorem 2.22]{MR2497314}.
%
Here $\varrho_p$ is a Hamiltonian function on $\hat{M}$ satisfying
\begin{enumerate}[(i)]
 \item
$\varrho_p < 0$ on $M$,
 \item \label{it:varrho2}
$\varrho_p$ goes to zero in the $C^2$ norm on $M$ as $p$ goes to infinity,
 \item \label{it:varrho3}
$\varrho_p$ depends only on the radial coordinate on the cylindrical end;
$$
 \varrho_p(x, r) = h_p(r),
 \qquad \forall (x, r) \in \partial M \times [1, \infty) \subset \hat{M}.
$$
 \item
$h_p'(r) \ge 0$ and $h_p''(r) \ge 0$ for all $r \in [1, \infty)$,
 \item
$h_p'(r) = p$ for $r \in [2, \infty)$, and
 \item \label{it:varrho6}
for any $K > 0$ and any $r \in (1, \infty)$,
there is an integer $N$ such that
$$
 r h_p'(r) - h_p(r) > K,
  \qquad \forall p > N.
$$
\end{enumerate}
The sequence $\{ \varrho_p \}_p$ is a cofinal family of admissible Hamiltonians,
so that the isomorphism \eqref{eq:isom1} comes from the definition of the wrapped Floer cohomology.

The condition \eqref{it:varrho2} implies that
for any $\epsilon > 0$,
the action of any Hamiltonian chord of $\varrho_p$ in $M$
is greater than $- \epsilon$ for sufficiently large $p$.
The condition \eqref{it:varrho3} implies that
the Hamiltonian vector field of $\varrho_p$ in the cylindrical end
is $h_p'(r)$ times the Reeb vector field on $\partial M$.
It follows that Hamiltonian chords of length one
are in one-to-one correspondence
with Reeb chords of length $h_p'(r)$,
and the action of a Hamiltonian chord
$(x, r) : [0,1] \to M \times [1, \infty)$ is given by
\begin{align*}
 A_{\varrho_p}(x, r)
  &= \int_0^1 \big( x^* \theta - \varrho(x, r)) d t \big)
  + f(x(1)) - f(x(0)), \\
  &= r h_p'(r) - h_p(r)
  + f(x(1)) - f(x(0)).
\end{align*}
The condition $\theta|_{\partial L} = 0$ implies
$\theta|_{\hat{L} \setminus L} = 0$,
so that the primitive function $f(x) = \int^x \theta$ is constant
on each connected component of $\hat{L} \setminus L$
and hence bounded.
Then the condition \eqref{it:varrho6} shows that
the actions of Hamiltonian chords on the cylindrical end are positive
for sufficiently large $p$.
As a result,
one obtains the isomorphism \eqref{eq:isom6},
where $\HF^*_{[-\epsilon,\infty)}(\hat{L}; \varrho_p)$ is
the subgroup of $\HF^*(\hat{L}; \varrho_p)$
generated by chords of action greater than or equal to $\epsilon$.

The construction of $K_p$ from $\varrho_p$ proceeds in two steps
\cite{MR1760631,MR2497314}:
First one modifies $\varrho_p$ to a Hamiltonian $\varsigma_p$
which is constant outside a large compact set $\kappa$
while only adding chords of action less than $- \epsilon$.
Then one adds to $\varsigma_p$ a Lefschetz admissible Hamiltonian $L_p$,
which is zero in the region $\kappa$
but has action bounded above,
so that Hamiltonian chords of $K_p := L_p + \varsigma_p$
outside $\kappa$ have action less than $- \epsilon$.
For a suitable choice of a family of admissible almost complex structures,
\begin{itemize}
 \item
there is a bijection between Hamiltonian chords of $K_p$ of action
greater than $- \epsilon$
and Hamiltonian chords of $\varrho_p$, and
 \item
this bijection inducing an isomorphism of the moduli spaces of Floer trajectories.
\end{itemize}
This gives the isomorphism \eqref{eq:isom2}.
The sequence $\{ G_p \}_p$ is a cofinal family of Lefschetz admissible Hamiltonians
such that
\begin{itemize}
 \item
there are sequences $p_i$ and $q_i$ of positive integers such that
$$
 K_{p_i} \le G_{q_i} \le K_{p_{i+1}}
$$
for all $i$, and
 \item
all Hamiltonian chords of $G_p$ have action greater than $-\epsilon$.
\end{itemize}
This induces the isomorphisms \eqref{eq:isom3} and \eqref{eq:isom4}.
The isomorphism \eqref{eq:isom5} comes from the cofinality of $\{ G_p \}_p$
and \pref{th:appendix} is proved.

\subsection{Fiber products of Lefschetz fibrations}

Let $\pi_1 : E_1 \to S$ and $\pi_2 : E_2 \to S$ be exact Lefschetz fibrations
and consider the fiber product
\begin{align*}
\begin{psmatrix}
 & E := E_1 \times_S E_2 \\
 E_1 & & E_2 \\
 & S
\end{psmatrix}.
\psset{nodesep=5pt,shortput=nab,arrows=->}
\ncline{1,2}{2,1}_{\tau_1}
\ncline{1,2}{2,3}^{\tau_2}
\ncline{2,1}{3,2}^{\pi_1}
\ncline{2,3}{3,2}_{\pi_2}
\ncline{1,2}{3,2}^{\pi}
\end{align*}
By smoothing corners,
we obtain a Liouville domain $M$ with a Liouville one-form
$
 \theta = \tau_1^* \Theta_1 + \tau_2^* \Theta_2 + k \pi^* \theta_S
$
for sufficiently large $k$
whose completion $\hat{M}$
is symplectomorphic to $\hat{E} := \hat{E_1} \times_{\hat{S}} \hat{E_2}$.
We have fiberwise cylindrical coordinates $r_{F_i}$, $i=1, 2$ and
a cylindrical coordinate $r_S$ on the base.
We say that a Hamiltonian $H : E \to \bR$ is {\em fibered admissible}
if
$$
 H=\pi^*H_S + \tau_1^* H_{F_1} + \tau_2^* H_{F_2}
$$
where
\begin{itemize}
 \item
$H_S$ is an admissible Hamiltonian on $\hat{S}$, and
 \item
$H_{F_i}$ is a Hamiltonian on $\hat{E_i}$ which is
\begin{itemize}
 \item
zero on $E_i \cup B_i$, and
 \item
a pull-back of an admissible Hamiltonian of $(F_i)_e$
on $A_i := (F_i)_e \times \hat{S}$.
\end{itemize}
\end{itemize}
Lagrangian submanifolds $L_i$ of $E_i$
fibered over a common path $\gamma : [0, 1] \to S$
gives a Lagrangian submanifold $L := L_1 \times_\gamma L_2$ of $E$,
which can be completed to a Lagrangian submanifold $\hat{L}$ of $\hat{E}$.
Although $\pi : E \to S$ is not a Lefschetz fibration
but a Bott-Morse analog of a Lefschetz fibration,
the proof of \pref{th:appendix} can be adapted
in a straightforward way to prove the following:

\begin{theorem} \label{th:appendix2}
One has an isomorphism
\begin{align} \label{eq:appendix2}
  \HW^*(\hat{L}) \cong \varinjlim_p \HF^*(\hat{L}; p H)
\end{align}
of rings.
\end{theorem}

The right hand side does not depend on the choice
of a fibered admissible Hamiltonian $H$,
or a cofinal family $\{ H_p \}_p$ of fibered admissible Hamiltonians in general.
One starts with a cofinal family
$\{ \varrho_p \}_p$ of admissible Hamiltonians,
truncate outside a large compact set $\kappa$
to obtain $\varsigma_p$,
then adds a fibered admissible Hamiltonian $L_p$
supported outside of $\kappa$
to obtain $K_p = \varsigma_p + L_p$.
This process can be performed
without changing chords with actions greater than $- \epsilon$,
and one obtains the isomorphism \eqref{eq:appendix2}.

\subsection{Symplectic cohomology and the bulk-boundary map}

In view of McLean's work, it is also natural to discuss the implication of the calculations in this paper for symplectic cohomology. Our treatment here is less detailed because the discussion which follows is complementary to our main topic.

\begin{theorem}
Let $L$ be an admissible Lagrangian
which is also a section of the SYZ fibration for the conifold.
Then we have an isomorphism of rings $SH^0(\hat{M}) \to WF(L)$.
\end{theorem}

\begin{proof} Consider a fibered admissible Hamiltonian $H$ as in the main part of this paper and assume that the discriminant points are generically positioned away from $L$ inside of $M$. For an appropriate choice of $H$ as above, Hamiltonian orbits are precisely:

\begin{itemize}
\item $T^3$ submanifolds on $\hat{M}$, which fiber over circles in the base
\item one-dimensional tori of orbits living in the fibers over the discriminant locus
\end{itemize}

Standard Morse-Bott theory allows one to find a perturbation which has exactly $2^m$ orbits corresponding to generators of $H^*(T^m)$ for each submanifold of Hamiltonian orbits.  For our purposes, it will be sufficient to consider the $T^3$ submanifolds which fiber over circles in the base because the orbits of the second type have grading at least 2.  We will be interested in the $SH^0(\hat{M})$, which is generated by the cochains arising from $H^0(T^3)$.

A priori there could be a differential $$\partial: SH^0(\hat{M}) \to SH^1(\hat{M})$$

However, curves contributing to this differential would necessarily preserve the free homotopy class of the projection of the chord to $\mathbb{C}^*$. An energy estimate similar to that in \cite[Theorem 18.10]{Seidel_CDST} shows that curves connecting Morse-Bott submanifolds which live in the same fiber must live entirely within the fiber and hence there are no such differentials. \vskip 10 pt

 The essential geometric idea which underlies our theorem is that generators of wrapped Floer homology between $L$ and $\phi_{H}(L)$ consist of a single Hamiltonian chord in each of the $T^3$ submanifolds, which we may think of as being geometrically the same as the corresponding generator in $SH^0(\hat{M})$. To turn this observation into a precise statement, we note that Abouzaid has defined a closed-open unital ring homomorphism
$$
 \mathcal{CO} :SH^*(\hat{M}) \to WF(L).
$$
This map is defined by counting solutions to a perturbed $J$-holomorphic curve equation with one interior puncture which is required to be asymptotic to our Hamiltonian orbit. \vskip 10 pt

In our setting this map can be completely calculated. More precisely the map sends the class in $H^0(T^3)$ to the unique Hamiltonian chord of $L$ in each submanifold worth of orbits. The non-trivial component of our map corresponds to ``local" curves, e.g. curves which do not escape some fixed neighborhood of the orbits. Using elementary Morse-Bott analysis, one can show that these correspond to the classical intersection $T^3 \cap L$. As before, there can be no non-trivial curves connecting different Morse-Bott submanifolds. In particular, this map induces an isomorphism
$$
 SH^0(\hat{M}) \to WF(L).
$$
 \end{proof}

Our computations in this paper therefore allow us to compute $SH^0(\hat{M})$ as well:

\begin{corollary} $ SH^0(\hat{M}) \cong \mathbb{C}[u,v,w_1,w_1^{-1},w_2,w_2^{-1}]/(uv=(1+w_1)(1+w_2)) $ \end{corollary}

\begin{remark} While finishing this paper, we noticed that Pascaleff \cite{Pascaleff_SMLCY} has very recently proven a similar theorem in his study of log Calabi-Yau surfaces. While our notion of Lagrangian section comes from an SYZ fibration, Pascaleff considers Lagrangian sections of an SYZ fibration defined only in a neighborhood of the compactifying divisor for log Calabi-Yau surfaces. It would be interesting to study the relationship between these approaches in more detail.  \end{remark}

 \bibliographystyle{amsalpha}
\bibliography{bibs}

\newcommand{\etalchar}[1]{$^{#1}$}
\def\cprime{$'$} \def\cprime{$'$}
\providecommand{\bysame}{\leavevmode\hbox to3em{\hrulefill}\thinspace}
\providecommand{\MR}{\relax\ifhmode\unskip\space\fi MR }
\providecommand{\MRhref}[2]{%
  \href{http://www.ams.org/mathscinet-getitem?mr=#1}{#2}
}
\providecommand{\href}[2]{#2}
\begin{thebibliography}{AGM{\etalchar{+}}00}

\bibitem[AAK]{Abouzaid-Auroux-Katzarkov_LFB}
Mohammed Abouzaid, Denis Auroux, and Ludmil Katzarkov, \emph{Lagrangian
  fibrations on blowups of toric varieties and mirror symmetry for
  hypersurfaces}, arXiv:1205.0053.

\bibitem[Abo10]{Abouzaid_GCGFC}
Mohammed Abouzaid, \emph{A geometric criterion for generating the {F}ukaya
  category}, Publ. Math. Inst. Hautes \'Etudes Sci. (2010), no.~112, 191--240.
  \MR{2737980}

\bibitem[Abo11]{MR2786590}
\bysame, \emph{A topological model for the {F}ukaya categories of plumbings},
  J. Differential Geom. \textbf{87} (2011), no.~1, 1--80. \MR{2786590
  (2012h:53193)}

\bibitem[AGM{\etalchar{+}}00]{MR1743597}
Ofer Aharony, Steven~S. Gubser, Juan Maldacena, Hirosi Ooguri, and Yaron Oz,
  \emph{Large {$N$} field theories, string theory and gravity}, Phys. Rep.
  \textbf{323} (2000), no.~3-4, 183--386. \MR{1743597 (2001c:81134)}

\bibitem[AJ10]{Akaho-Joyce}
Manabu Akaho and Dominic Joyce, \emph{Immersed {L}agrangian {F}loer theory}, J.
  Differential Geom. \textbf{86} (2010), no.~3, 381--500. \MR{2785840
  (2012e:53180)}

\bibitem[Aka05]{MR2155230}
Manabu Akaho, \emph{Intersection theory for {L}agrangian immersions}, Math.
  Res. Lett. \textbf{12} (2005), no.~4, 543--550. \MR{2155230 (2006h:53093)}

\bibitem[AP01]{Arinkin-Polishchuk_FCFT}
D.~Arinkin and A.~Polishchuk, \emph{Fukaya category and {F}ourier transform},
  Winter {S}chool on {M}irror {S}ymmetry, {V}ector {B}undles and {L}agrangian
  {S}ubmanifolds ({C}ambridge, {MA}, 1999), AMS/IP Stud. Adv. Math., vol.~23,
  Amer. Math. Soc., Providence, RI, 2001, pp.~261--274. \MR{1876073
  (2002k:32040)}

\bibitem[AS10]{Abouzaid-Seidel_OSAVF}
Mohammed Abouzaid and Paul Seidel, \emph{An open string analogue of {V}iterbo
  functoriality}, Geom. Topol. \textbf{14} (2010), no.~2, 627--718. \MR{2602848
  (2011g:53190)}

\bibitem[Aur07]{Auroux_MSTD}
Denis Auroux, \emph{Mirror symmetry and {$T$}-duality in the complement of an
  anticanonical divisor}, J. G\"okova Geom. Topol. GGT \textbf{1} (2007),
  51--91. \MR{2386535 (2009f:53141)}

\bibitem[Aur09]{Auroux_SLFWCMS}
\bysame, \emph{Special {L}agrangian fibrations, wall-crossing, and mirror
  symmetry}, Surveys in differential geometry. {V}ol. {XIII}. {G}eometry,
  analysis, and algebraic geometry: forty years of the {J}ournal of
  {D}ifferential {G}eometry, Surv. Differ. Geom., vol.~13, Int. Press,
  Somerville, MA, 2009, pp.~1--47. \MR{MR2537081}

\bibitem[Aur15]{AurouxInfinite}
\bysame, \emph{Infinitely many monotone {L}agrangian tori in $\mathbb{R}^6$},
  Invent. Math. \textbf{201} (2015), no.~3, 909--924.

\bibitem[BK90]{Bondal-Kapranov_ETC}
A.~I. Bondal and M.~M. Kapranov, \emph{Enhanced triangulated categories}, Mat.
  Sb. \textbf{181} (1990), no.~5, 669--683. \MR{MR1055981 (91g:18010)}

\bibitem[BK04]{Bezrukavnikov-Kaledin_2004}
R.~V. Bezrukavnikov and D.~B. Kaledin, \emph{Mc{K}ay equivalence for symplectic
  resolutions of quotient singularities}, Tr. Mat. Inst. Steklova \textbf{246}
  (2004), no.~Algebr. Geom. Metody, Svyazi i Prilozh., 20--42. \MR{MR2101282
  (2006e:14006)}

\bibitem[Bon89]{Bondal_RAACS}
A.~I. Bondal, \emph{Representations of associative algebras and coherent
  sheaves}, Izv. Akad. Nauk SSSR Ser. Mat. \textbf{53} (1989), no.~1, 25--44.
  \MR{MR992977 (90i:14017)}

\bibitem[BvdB03]{Bondal-van_den_Bergh}
A.~Bondal and M.~van~den Bergh, \emph{Generators and representability of
  functors in commutative and noncommutative geometry}, Mosc. Math. J.
  \textbf{3} (2003), no.~1, 1--36, 258. \MR{1996800 (2004h:18009)}

\bibitem[Cha13]{Chan_HMSARTD}
Kwokwai Chan, \emph{Homological mirror symmetry for {$A_n$}-resolutions as a
  {$T$}-duality}, J. Lond. Math. Soc. (2) \textbf{87} (2013), no.~1, 204--222.
  \MR{3022713}

\bibitem[Cho04]{Cho_CT}
Cheol-Hyun Cho, \emph{Holomorphic discs, spin structures, and {F}loer
  cohomology of the {C}lifford torus}, Int. Math. Res. Not. (2004), no.~35,
  1803--1843. \MR{2057871 (2005a:53148)}

\bibitem[CLL12]{Chan-Lau-Leung12}
Kwokwai Chan, Siu-Cheong Lau, and Naichung~Conan Leung, \emph{S{YZ} mirror
  symmetry for toric {C}alabi-{Y}au manifolds}, J. Differential Geom.
  \textbf{90} (2012), no.~2, 177--250. \MR{2899874}

\bibitem[CO06]{Cho-Oh}
Cheol-Hyun Cho and Yong-Geun Oh, \emph{Floer cohomology and disc instantons of
  {L}agrangian torus fibers in {F}ano toric manifolds}, Asian J. Math.
  \textbf{10} (2006), no.~4, 773--814. \MR{MR2282365 (2007k:53150)}

\bibitem[CU13]{Chan-Ueda_DTFHMSAN}
Kwokwai Chan and Kazushi Ueda, \emph{Dual torus fibrations and homological
  mirror symmetry for {$A_n$}-singlarities}, Commun. Number Theory Phys.
  \textbf{7} (2013), no.~2, 361--396. \MR{3164868}

\bibitem[DS15]{Donovan-Segal_MBG}
Will Donovan and Ed~Segal, \emph{Mixed braid group actions from deformations of
  surface singularities}, Comm. Math. Phys. \textbf{335} (2015), no.~1,
  497--543. \MR{3314511}

\bibitem[Dui80]{Duistermaat_GAAC}
J.~J. Duistermaat, \emph{On global action-angle coordinates}, Comm. Pure Appl.
  Math. \textbf{33} (1980), no.~6, 687--706. \MR{596430 (82d:58029)}

\bibitem[FOOO09]{Fukaya-Oh-Ohta-Ono}
Kenji Fukaya, Yong-Geun Oh, Hiroshi Ohta, and Kaoru Ono, \emph{Lagrangian
  intersection {F}loer theory: anomaly and obstruction}, AMS/IP Studies in
  Advanced Mathematics, vol.~46, American Mathematical Society, Providence, RI,
  2009. \MR{MR2553465}

\bibitem[FOOO10]{FOOO_toric_I}
\bysame, \emph{Lagrangian {F}loer theory on compact toric manifolds. {I}}, Duke
  Math. J. \textbf{151} (2010), no.~1, 23--174. \MR{2573826}

\bibitem[FU10]{Futaki-Ueda_A-infinity}
Masahiro Futaki and Kazushi Ueda, \emph{Exact {L}efschetz fibrations associated
  with dimer models}, Math. Res. Lett. \textbf{17} (2010), no.~6, 1029--1040.
  \MR{2729627}

\bibitem[GL87]{Geigle-Lenzing_WPC}
Werner Geigle and Helmut Lenzing, \emph{A class of weighted projective curves
  arising in representation theory of finite-dimensional algebras},
  Singularities, representation of algebras, and vector bundles (Lambrecht,
  1985), Lecture Notes in Math., vol. 1273, Springer, Berlin, 1987,
  pp.~265--297. \MR{MR915180 (89b:14049)}

\bibitem[Gro01]{MR1882328}
Mark Gross, \emph{Examples of special {L}agrangian fibrations}, Symplectic
  geometry and mirror symmetry ({S}eoul, 2000), World Sci. Publ., River Edge,
  NJ, 2001, pp.~81--109. \MR{1882328 (2003f:53085)}

\bibitem[Her00]{MR1760631}
David Hermann, \emph{Holomorphic curves and {H}amiltonian systems in an open
  set with restricted contact-type boundary}, Duke Math. J. \textbf{103}
  (2000), no.~2, 335--374. \MR{1760631 (2001b:53112)}

\bibitem[Kel01]{Keller_IAAM}
Bernhard Keller, \emph{Introduction to {$A$}-infinity algebras and modules},
  Homology Homotopy Appl. \textbf{3} (2001), no.~1, 1--35 (electronic).
  \MR{MR1854636 (2004a:18008a)}

\bibitem[KKV97]{MR1467889}
Sheldon Katz, Albrecht Klemm, and Cumrun Vafa, \emph{Geometric engineering of
  quantum field theories}, Nuclear Phys. B \textbf{497} (1997), no.~1-2,
  173--195. \MR{1467889 (98h:81097)}

\bibitem[LYZ00]{Leung-Yau-Zaslow_SLHYM}
Naichung~Conan Leung, Shing-Tung Yau, and Eric Zaslow, \emph{From special
  {L}agrangian to {H}ermitian-{Y}ang-{M}ills via {F}ourier-{M}ukai transform},
  Adv. Theor. Math. Phys. \textbf{4} (2000), no.~6, 1319--1341. \MR{1894858
  (2003b:53053)}

\bibitem[McL09]{MR2497314}
Mark McLean, \emph{Lefschetz fibrations and symplectic homology}, Geom. Topol.
  \textbf{13} (2009), no.~4, 1877--1944. \MR{2497314 (2011d:53224)}

\bibitem[Mer99]{Merkulov_SHAKM}
S.~A. Merkulov, \emph{Strong homotopy algebras of a {K}\"ahler manifold},
  Internat. Math. Res. Notices (1999), no.~3, 153--164. \MR{MR1672242
  (2000h:32026)}

\bibitem[Mil63]{MR0157388}
J.~Milnor, \emph{Spin structures on manifolds}, Enseignement Math. (2)
  \textbf{9} (1963), 198--203. \MR{0157388 (28 \#622)}

\bibitem[Orl11]{Orlov_FCIC}
Dmitri Orlov, \emph{Formal completions and idempotent completions of
  triangulated categories of singularities}, Adv. Math. \textbf{226} (2011),
  no.~1, 206--217. \MR{2735755}

\bibitem[Pas]{Pascaleff_SMLCY}
James Pascaleff, \emph{On the symplectic cohomology of log {C}alabi-{Y}au
  surfaces}, arXiv:1304.5298.

\bibitem[Pas14]{Pascaleff_FCMPP}
\bysame, \emph{Floer cohomology in the mirror of the projective plane and a
  binodal cubic curve}, Duke Math. J. \textbf{163} (2014), no.~13, 2427--2516.
  \MR{3265556}

\bibitem[Ric89]{Rickard}
Jeremy Rickard, \emph{Morita theory for derived categories}, J. London Math.
  Soc. (2) \textbf{39} (1989), no.~3, 436--456. \MR{MR1002456 (91b:18012)}

\bibitem[Rit13]{Ritter_TQFTSC}
Alexander Ritter, \emph{Topological quantum field theory structure on
  symplectic cohomology}, J. Topol. \textbf{6} (2013), no.~2, 391--489.
  \MR{3065181}

\bibitem[Sei]{Seidel_CDST}
Paul Seidel, \emph{Categorical dynamics and symplectic topology}, lecture notes
  for the graduate course at {MIT} in Spring 2013, available at
  {\texttt{http://www-math.mit.edu/{\textasciitilde}seidel/937/index.html}}.

\bibitem[Sei08a]{Seidel_ASNT}
\bysame, \emph{{$A\sb \infty$}-subalgebras and natural transformations},
  Homology, Homotopy Appl. \textbf{10} (2008), no.~2, 83--114. \MR{2426130
  (2010k:53154)}

\bibitem[Sei08b]{Seidel_PL}
\bysame, \emph{Fukaya categories and {P}icard-{L}efschetz theory}, Zurich
  Lectures in Advanced Mathematics, European Mathematical Society (EMS),
  Z\"urich, 2008. \MR{MR2441780}

\bibitem[Sei10]{Seidel_suspension}
\bysame, \emph{Suspending {L}efschetz fibrations, with an application to local
  mirror symmetry}, Comm. Math. Phys. \textbf{297} (2010), no.~2, 515--528.
  \MR{2651908}

\bibitem[Sei11]{Seidel_K3}
\bysame, \emph{Homological mirror symmetry for the quartic surface},
  math.AG/0310414, 2011.

\bibitem[ST01]{Seidel-Thomas}
Paul Seidel and Richard Thomas, \emph{Braid group actions on derived categories
  of coherent sheaves}, Duke Math. J. \textbf{108} (2001), no.~1, 37--108.
  \MR{MR1831820 (2002e:14030)}

\bibitem[SYZ96]{Strominger-Yau-Zaslow}
Andrew Strominger, Shing-Tung Yau, and Eric Zaslow, \emph{Mirror symmetry is
  {$T$}-duality}, Nuclear Phys. B \textbf{479} (1996), no.~1-2, 243--259.
  \MR{MR1429831 (97j:32022)}

\bibitem[TU10]{Toda-Uehara}
Yukinobu Toda and Hokuto Uehara, \emph{Tilting generators via ample line
  bundles}, Adv. Math. \textbf{223} (2010), no.~1, 1--29. \MR{2563209}

\bibitem[VdB04]{Van_den_Bergh_TFNR}
Michel Van~den Bergh, \emph{Three-dimensional flops and noncommutative rings},
  Duke Math. J. \textbf{122} (2004), no.~3, 423--455. \MR{MR2057015
  (2005e:14023)}

\end{thebibliography}

\noindent
Kwokwai Chan

Department of Mathematics,
The Chinese University of Hong Kong,
Shatin,
Hong Kong

{\em e-mail address}\ :\ kwchan@math.cuhk.edu.hk
\ \vspace{0mm} \\

\noindent
Daniel Pomerleano

Kavli Institute for the Physics and Mathematics of the Universe,
University of Tokyo, 5-1-5 Kashiwanoha, Kashiwa, 277-8583, Japan.

{\em e-mail address}\ :\ daniel.pomerleano@gmail.com
\ \vspace{0mm} \\

\noindent
Kazushi Ueda

Graduate School of Mathematical Sciences,
The University of Tokyo,
3-8-1 Komaba,
Meguro-ku,
Tokyo,
153-8914,
Japan.

{\em e-mail address}\ : \  kazushi@ms.u-tokyo.ac.jp
\ \vspace{0mm} \\

\end{document}